\documentclass[preprint]{imsart}

\usepackage[T1]{fontenc}

\usepackage{amsmath}
\usepackage{amsthm}
\usepackage{amssymb}
\PassOptionsToPackage{normalem}{ulem}
\usepackage{ulem}
\usepackage{bbm}
\usepackage{enumitem}
\setitemize{leftmargin=*} 
\usepackage{graphicx}
\usepackage{comment}
\usepackage{xcolor}
\usepackage{harvard}
\usepackage{lmodern}

\newcommand\bse{\begin{subequations}}
\newcommand\ese{\end{subequations}}
\newcommand\be{\begin{equation}}
\newcommand\ee{\end{equation}}
\newcommand\bee{\begin{equation*}}
\newcommand\eee{\end{equation*}}
\newcommand\ba{\begin{align}}
\newcommand\ea{\end{align}}
\newcommand\baa{\begin{align*}}
\newcommand\eaa{\end{align*}}
\newcommand\bi{\begin{itemize}}
\newcommand\ei{\end{itemize}}
\newcommand\ben{\begin{enumerate}}
\newcommand\een{\end{enumerate}}

\newcommand\bpm{\begin{pmatrix}}
\newcommand\epm{\end{pmatrix}}

\newcommand{\R}{\mathbb{R}}
\newcommand{\Z}{\mathbb{Z}}

\newcommand{\e}{\mathrm{e}}

\newcommand{\ind}{\mathbbm{1}}

\newcommand{\Cov}{\mbox{Cov}}

\def\Cov{\mathrm{Cov}}

\def\E{\mathrm{E}}

\includecomment{mycomment}

\newtheorem{thm}{Theorem}

\newtheorem{remark}{Remark}
\newtheorem{lemma}{Lemma}
\newtheorem{proposition}{Proposition}

\theoremstyle{definition}

\theoremstyle{remark}

\def\R{\mbox{$\mathbb R$}}
\def\N{\mbox{$\mathbb N$}}
\def\Z{\mbox{$\mathbb Z$}}

\newcommand{\bd}{\begin{displaymath}}
\newcommand{\ed}{\end{displaymath}}
\newcommand{\bea}{\begin{eqnarray}}
\newcommand{\eea}{\end{eqnarray}}
\newcommand{\bean}{\begin{eqnarray*}}
\newcommand{\eean}{\end{eqnarray*}}

\begin{document}

\setcounter{page}{1}
\begin{frontmatter}
\title{
Gaussian Approximation for  Lag-Window Estimators 
and  the Construction of Confidence Bands for the Spectral Density
}
\runtitle{Gaussian Approximation for Spectral Density Estimators}
\begin{aug}
\author{
	\fnms{Jens-Peter} \snm{Krei{\ss}} \thanksref{a} 
\corref{} 
}
\hspace{-.7em}
\author{
	\fnms{Anne} \snm{Leucht}  \thanksref{b} 
\corref{} 
}
\hspace{-.7em}
\author{
\fnms{Efstathios } \snm{Paparoditis} \thanksref{c} 
\corref{} 
}

\address[a]{Institut f\"ur Mathematische Stochastik, TU Braunschweig, 38106 Braunschweig, Germany. E-mail:
j.kreiss@tu-braunschweig.de}

\address[b]{Institut f\"ur Statistik, Universit\"at Bamberg, 96052 Bamberg, Germany. E-mail: anne.leucht@uni-bamberg.de}

\address[c]{Cyprus Academy of Sciences, Letters and Arts,  P.O.Box 22554, CY-1522 Nicosia, Cyprus. E-mail:
stathisp@ucy.ac.cy}

\runauthor{Krei{\ss}, Leucht, Paparoditis }
\affiliation{TU Braunschweig, Universit\"at Bamberg, University of Cyprus}
\end{aug}

\begin{abstract}
In this paper we consider the construction of simultaneous confidence bands for the spectral density
of a stationary time series using a Gaussian approximation for classical lag-window spectral density estimators evaluated
at the set of all positive Fourier frequencies. The Gaussian approximation opens up the possibility to verify asymptotic validity of a multiplier bootstrap procedure and, even further, to derive the corresponding rate of convergence. A small simulation study sheds light on the finite sample properties of this bootstrap proposal. 
\end{abstract}




%
%

\end{frontmatter}

\section{Introduction}\label{Section_Introduction}

We consider the problem of constructing
(simultaneous)
 confidence bands  for the spectral density of a stationary time series.  Toward this goal we  develop   Gaussian approximation results  for  the maximum deviation over all (positive) Fourier frequencies of a lag-window spectral density estimator and for its bootstrap counterpart.  Based on observations $X_1, X_2, \ldots , X_T$ stemming from a strictly stationary and centered stochastic process  $\{X_t,t\in\Z\}$, a lag-window  estimator of the spectral density $ f(\lambda)$, $\lambda\in [0,\pi]$,  is given by  
\begin{equation}
\label{spec_dens}
\widehat f_T(\lambda) =\frac{1}{2\pi} \sum _{\vert j\vert \le M_T} w(j/M_T)\, \e ^{-i j \lambda}\, \widehat \gamma (j), 
~ \lambda \in [0,\pi ].
\end{equation}
In \eqref{spec_dens}  and for $j=0, 1, 2, \ldots, M_T<T$,
\begin{equation} \label{eq.gamma-hat}
\widehat \gamma (j) = \frac{1}{T} \sum _{t=j+1}^{T} X_t\,X_{t-j}
\end{equation}
 are  estimators of the autocovariances $ \gamma(j)=\Cov(X_0,X_j)$ of the process $\{X_t,t\in\Z\}$ and $\widehat{\gamma}(j)=\widehat{\gamma}(-j)$ for $ j<0$.  The function  $w:[-1,1] \rightarrow {\mathbb R}$ is a so-called lag-window,
 which assigns weights to the $M_T$ sample autocovariances effectively 
used in  the calculation of the estimator for  
$ f(\lambda)$ and which satisfies some assumptions to be specified later.

It is well-known that for a fixed frequency $\lambda$, under suitable assumptions on the dependence structure of the underlying time series and for $M_T$ 
converging not too fast to infinity, asymptotic normality for lag-window estimators
can be shown. To elaborate and assuming sufficient smoothness of $f$,
which is equivalent to assuming a sufficiently fast decay of the underlying autocovariances $\gamma (h)$ as $|h|\to \infty$,
one typically can show under certain conditions on the lag-window $w$ and for $M_T\to \infty $ such that $T/M_T^5 \to C^2\ge 0$ as $n\to \infty$, that
\begin{equation}
\label{spec_dens_normal}
\sqrt{\frac{T}{M_T}}\Big( \widehat f_T(\lambda ) -f(\lambda )\Big) \xrightarrow{\mathcal D} 
{\mathcal N} \Big( C\, W \, f^{\prime \prime}(\lambda ),\, f^2(\lambda ) \int _{-1}^1 w^2(u)\, du\Big),
\end{equation}
 for $0<  \lambda <\pi $, where $f^{\prime\prime}$ denotes the second derivative of $f$ and $W=\lim_{u\rightarrow 0}(1-w(u))/u^2$, where the latter is assumed to exist and to be positive.   

Moreover, for the estimator $\widehat f_T(\lambda)$ with sophisticated selected so-called flat-top lag-windows
$w$ and under the additional assumption that $ \sum _{j\in \mathbb{Z}} |j|^r |\gamma (j)| < \infty $ for some $r\ge 1$
convergence rates for the mean squared error $\mbox{MSE}(\widehat f_T(\lambda)) = O(T^{-2r/(2r+1)})$ can 
be achieved, which  corresponds to $M_T\sim T^{1/(2r+1)}$. See  \citeasnoun{PolitisRomano87} and \citeasnoun{BergPolitis2009} for details.

Instead of point-wise inference we aim in this paper for simultaneous 
inference using  the  estimator $\widehat{f}_T$. 
More precisely, we aim for confidence bands covering
the spectral density uniformly at all
positive Fourier frequencies $\lambda _{k,T} = 2\pi k/T, k=1, \ldots , \lfloor T/2 \rfloor =:N_T$ with a desired (high)  probability. Hence, we keep all  information contained  in $X_1,\dots, X_T$ in the sense that the discrete Fourier transform of any vector $x\in\R^T$ is a linear combination of trigonometric polynomials at exactly these frequencies.

Note that so far simultaneous confidence bands for the spectral density had only been derived for special cases. In particular, for  autoregressive processes relying on parametric spectral density estimates (see e.g.\;\citeasnoun{NP84} and \citeasnoun{T87})  while  \citeasnoun{NP08} proposed a bootstrap-aided approach for Gaussian time series using the integrated periodogram to construct simultaneous confidence bands for (a smoothed version of) the spectral density. Generalizing results from \citeasnoun{LiuWu2010}, \citeasnoun{YZ22} derived simultaneous confidence bands  for the spectral density on a grid with mesh size wider than $2\pi/T$ and for  locally stationary Bernoulli shifts satisfying a geometric moment contraction condition. To this end, they proved that the maximum deviation of a
suitably standardized lag-window spectral density estimator
over such a grid
asymptotically possesses a Gumbel distribution. Motivated by the slow rate of  convergence of the 
maximum deviation to a Gumbel variable, they propose an asymptotically valid bootstrap method to improve the finite 
sample behavior of their confidence regions.

As mentioned above, we aim to construct simultaneous confidence bands on the finer grid of all (positive) Fourier frequencies. 
To achieve this goal, we establish 
a Gaussian approximation result for the distribution of a properly standardized statistic based on  the random quantity
$$
\max _{1\leq k\leq N_T} \vert \widehat f_T(\lambda_{k,T}) - E \widehat f_T(\lambda_{k,T})\vert .
$$ 
In this context and in order to stabilize the asymptotic variance and to construct confidence bands that automatically 
adapt to the local variability of $f(\lambda)$, we   are particularly  interested  in deriving a Gaussian approximation
for the normalized  statistic
$$
\max _{1\leq k \leq N_T} \frac{\displaystyle \vert \widehat f_T(\lambda_{k,T}) - E\widehat f_T(\lambda_{k,T})\vert }
{\displaystyle \widehat f_T(\lambda_{k,T})}.
$$ 
Various Gaussian approximations of max-type statistics in the time domain have been derived during the last decades. For means of high dimensional time series data we refer the reader to  
	\citeasnoun{ZW17} as well as \citeasnoun{ZC18} and references therein. In the context of i.i.d. functional data, \citeasnoun{CCH22} show that Gaussian approximation results for means of high dimensional vectors can be adapted to establish a Gaussian approximation for the maximum of the periodogram.  To elaborate, they use that the periodogram at frequency $\lambda_{k,T}$ is the squared norm of $\frac{1}{\sqrt{T2\pi}}\sum_{t=1}^T X_t\, e^{-it\lambda_{k,T}}$. However, their arguments do not directly carry over to  lag-window estimators as the latter have a more complex structure.   It is shown in \eqref{eq.Z} below, that they still have a mean-type representation, but the degree of dependence of the summands increases with the sample size which prevents the application of the afore-mentioned results.
  A Gaussian approximation result  for high dimensional spectral density matrices  of $\alpha$-mixing time series has been derived in \citeasnoun{Changetal23}. However, they  require the dimension of the process to increase with sample size making their  results  not applicable for our purposes.

The paper is organized as follows. Section 2 summarizes key assumptions for our results. The core
material on Gaussian approximation for spectral density estimators is contained in Section 3. In Section 4  we discuss 
application of the Gaussian approximation results to the construction of simultaneous confidence bands over the positive Fourier frequencies. Section~5 reports  the results of a small  simulation study. 
Auxiliary lemmas and proofs are deferred to  Section~6.

\section{Set up and Assumptions}\label{Section_Notation}

We begin by imposing the following assumption on the dependence structure of the  stochastic process generating the observed time series $X_1, X_2, \ldots, X_T$.
\vspace*{0.2cm}

\noindent {\bf Assumption 1:}  $ \{ X_t, t\in \Z\}$ is a strictly stationary, centered process and $X_t = g(e_t, e_{t-1}, \ldots ) $ 
for some measurable function $g$ and an i.i.d. process  $\{e_t,t\in\Z\}$  with mean  zero and 
variance $0< \sigma ^2_e<\infty$.   For $s\geq 0$, denote by   $\mathcal F_{r,s}=\sigma(e_r, e_{r-1},\dots, e_{r-s}) $ the 
$\sigma$-algebra generated by the  random variables  
$ \{e_{r-j},  0\leq j \leq s\}$ 
and assume 
that $E\vert X_t\vert ^m < \infty $ for some  $m >16$. For  a random variable $ X$ let  $ \|X\|_m:=(E|X|^m)^{1/m}$ and define  
for an independent copy  $\{e_t^{\prime},t\in\Z\}$ of $\{e_t;t\in\Z\}$,
 \[ 
 \delta _m(k):= \Vert X_t-g(e_t, e_{t-1}, \ldots, e_{t-k+1}, e^{\prime }_{t-k}, e_{t-k-1}, \ldots ) \Vert _m.
 \] 
The assumption is that
\be
\label{delta} 
\delta_m(k) \leq C\, (1+k)^{-\alpha}
\ee
for some  $\alpha >3$
and a constant $C<\infty$.
\vspace*{0.2cm}

Assumption 1 implies that the autocovariance function $ \gamma(h)=\Cov(X_0,X_h),$ $h\in\Z,$ is absolute summable, that is that  the process $ \{X_t,t\in\Z\}$ possesses a continuous and bounded spectral density
$ f(\lambda)= (2\pi)^{-1}\sum_{h\in\Z} \gamma(h) e^{-i h \lambda}$, $\lambda \in (-\pi,\pi]$.\\
It can be even shown 
that \eqref{delta} implies $ \sum _{j\in \mathbb{Z}} |j|^r |\gamma (j)| < \infty $ 
for $r<\alpha -1$, see Lemma~\ref{l.cov} in Section 6.
As has been mentioned in the Introduction this allows for the option of convergence rates for 
$\widehat f_T(\lambda) $ of order $O(T^{-r/(2r+1)})$, $r<\alpha -1$. Note that our assumptions on the decay of the dependence coefficients are less restrictive than those in  \citeasnoun{YZ22}, where exponentially decaying coefficients are presumed.
We assume  for simplicity that $EX_t=0$. Our results can be generalized to non-centered time series using the modified autocovariance estimator 
$\widetilde \gamma (j) = \frac{1}{T} \sum _{t=j+1}^{T} (X_t-\overline{X})(X_{t-j}-\overline {X}) \ \   \mbox{with} \ \   \overline{X}= \sum_{t=1}^T X_t/T$, 
instead of  the estimator $\widehat\gamma(j)$ defined in~\eqref{eq.gamma-hat}.\\

Additional to Assumption 1 we also require the following boundedness condition  for the spectral density~$f$.
\vspace*{0.2cm}

\noindent
{\bf Assumption 2:} The spectral density  $f$ satisfies $ \inf_{\lambda \in (-\pi,\pi]}f(\lambda) >0$. 
\vspace*{0.2cm}

\noindent Finally, we impose the following  conditions on the lag-window function~$w$ used in obtaining the estimator $ \widehat{f}_T$ and which are standard in the literature; see  \citeasnoun{Priestley81}, Chapter 6. 

\vspace*{0.2cm}

\noindent
{\bf Assumption 3:} The lag-window $w\colon[-1,1]\to\R$ is assumed to be a differentiable and symmetric function with  
$\int_{-1}^1 w(u)\, du=1$ and $w(0)=1$.   
\vspace*{0.2cm}
%

\section{Gaussian Approximation  for  Spectral Density Estimators}\label{Section_Results}

For the following considerations and in order to separate  any bias related problems, 
we consider the centered sequence  
\begin{align} \label{eq.Z}
\sqrt{\frac{T}{M_T}} \Big( \widehat f_T(\lambda_{k,T} ) -E\, \widehat f_T(\lambda_{k,T} )\Big) & =\frac{1}{\sqrt{T}}
\sum_{t=1}^T \sum _{j=0}^{M_T}  a_{k,j} \big( X_t X_{t-j} -\gamma (j)\big) \,\ind_{{t>j}}
\nonumber\\
&
=:\frac{1}{\sqrt{T}} \sum _{t=1}^T Z_{t,k},
\end{align}
with an obvious abbreviation for $ Z_{t,k}$ and  where  for $ \ k=1,\ldots , N_T$,
$$
a_{k,j}:= \begin{cases}
     \frac{\displaystyle 1}{\displaystyle 2\pi} \cdot  \frac{\displaystyle 1}{\displaystyle \sqrt{M_T}} &,\; j=0 \\
        \frac{\displaystyle 1}{\displaystyle \pi}  \cdot \frac{\displaystyle 1}{\displaystyle \sqrt{M_T}}w(j/M_T) \cos (j\lambda _{k,T})& ,\; j\ge 1.
\end{cases} 
$$
Due to the  differentiability of the  lag-window $w$ (cf. Assumption~3), we have 
\begin{align*}
\max _{k=1,\ldots , N_T} \sum _{j=0}^{M_T} a_{k,j}^2 
 & \le   \frac{1}{4\pi^2} \frac{1}{M_T} + \frac{1}{\pi^2 M_T} \sum _{j=1}^{M_T} w^2 (j/M_T) \\
&  = \frac{1}{\pi^2} \int_0^1 w^2(u)\, du +{\mathcal O}(M_T^{-1}).
\end{align*}
The  following theorem is our first result and establishes a  valid Gaussian approximation  for  the maximum over all positive Fourier frequencies  of  the  centered estimator  given in (\ref{spec_dens}). 

\begin{thm}\label{t.asymptotics}
Suppose that $\{X_t, t\in \mathbb Z\}$ fulfils  Assumption 1 and 2.  Let 
$\widehat{f}_T$ be a lag-window estimator of $f$ as given  in (\ref{spec_dens}),  where the lag-window $w$ satisfies  Assumption 3 and let 
    $M_T\rightarrow \infty $
and $M_T\sim T^{a_s}$.  Assume  further that 
\bea\label{eq.var_f}
&&\Big\Vert \,\frac{1}{\sqrt{T}}  \sum _{t=1}^T Z_{t,k} \,\Big\Vert _2^2 
>c>0.
\eea\\
Let 
$\xi _k, k=1, \ldots , N_T$, be  jointly normally distributed random variables with zero mean and covariance 
$E \xi_{k_1} \xi_{k_2}$ equal to
$$ 
\frac{1}{T}E\Big[\Big( \sum _{j=0}^{M_T} \sum _{t=j+1}^T  a_{k_1,j} \big( X_t X_{t-j} -\gamma (j)\big) \Big) 
\Big(  \sum _{j=0}^{M_T} \sum _{t=j+1}^T  a_{k_2,j} \big( X_t X_{t-j} -\gamma (j)\big) \Big)\Big].   
$$
Let  $ \lambda $, $ a_s $, and $ a_l  $   be positive constants    such that 
\begin{equation} \label{eq.a_s}
\alpha>\min\Big\{1+\frac{a_l}{2a_s}\,,\,\frac{1-a_l-\frac{12}{m}-4\lambda-\max\{\frac{4}{m},\, 2\lambda\}}{2a_s}+\frac{3}{2}\Big\}
 \end{equation}
and
\begin{equation} \label{eq.all}
0<a_s + \max\{4\lambda, 2\lambda +4/m\}  <  a_l < 1 -6\lambda -12/m.
\end{equation}
Then,   
\begin{align}  \label{thm_gauss}
&\sup _{x\in \mathbb R} \Big\vert P\big( \max _{k=1,\ldots , N_T}  
    \sqrt{\frac{T}{M_T}} \Big\vert \widehat f_T(\lambda_{k,T} )   -E\, \widehat f_T(\lambda_{k,T} )\Big\vert \le x \big)
    -  P\big(    \max _{k=1,\ldots , N_T} \vert \xi_k \vert \le x \big) \Big\vert\nonumber\\
     &  = {\mathcal O}\Big(T^{-\kappa} + T^{-\lambda}\big(\log(N_T)\big)^{3/2}  \Big) , 
\end{align}
where $ \kappa=\min\{\kappa_1,\kappa_2\}$ with 
\[ \kappa_1= a_l/2 -a_s/2-\lambda -\max\{2/m \,,\, \lambda\}  \ \mbox{and} \ \  \kappa_2 =1/2-a_l/2 -6/m -3\lambda.\] 
\end{thm}
\medskip

\noindent
Some remarks regarding  Theorem~\ref{t.asymptotics} are in order.

\begin{remark}{~}
	\begin{itemize}
		\item[(i)]
 The lower bound in  \eqref{eq.var_f} is valid for sufficiently large $T$. This results from our assumption that $f$ is bounded away from zero (see Assumption 2), together with 
  $$
 \begin{aligned}
&\sup_{k\in\{1,\dots, N_T \}} \Big|\big\Vert \frac{1}{\sqrt{T}}  \sum _{t=1}^T Z_{t,k} \big\Vert_2^2 - f^2(\lambda _{k,T}) \,\frac{1}{  M_T} 	\sum_{j=-M_T }^{M_T } 
w^2\big(\frac{j}{M_T}\big) \,(1+\cos(2j\,\lambda_{k,T}))  \Big|\\
&=o(1),
\end{aligned}
 $$
 see Lemma~\ref{l.inequalities} in Section 6.
 \item[(ii)] In order to verify that the max-statistic of interest can be approximated by the maximum of absolute values of Gaussian random variables inheriting the covariance structure of lag-window estimators, we combine smoothing techniques, Lindeberg's method, and a blocking approach. The parameter $\lambda$ appearing in Theorem~\ref{t.asymptotics} is related to the smoothness of the functions used in the proof to approximate indicator functions. The parameters $a_l$ and $a_s$ determine the sizes of the big and small blocks. The sizes of these blocks are tailor-made to handle the growing degree of dependence in the $Z_{t,k}'$s with increasing sample size.
 \end{itemize}
\end{remark}
\medskip

\begin{remark} \label{r.1}
{~}
\begin{itemize}
\item[(i)]
Notice first that the rate in the Gaussian approximation \eqref{thm_gauss} does not depend on the rate of decay of the dependence coefficients. Moreover, as can  be seen by an inspection of the proof of Theorem 1, even if the underlying time series $\{X_t, t \in \Z\}$ 
consists of i.i.d.~observations, i.e. $X_t=e_t$ for all $t$, we do not achieve better approximation rates using our method of proof. 
\item[(ii)] If we assume that the number $m$ of moments assumed to exist and the convergence rate for the underlying lag-window estimator $a _s=1/(2r+1),\; r\in \mathbb N,$  are fixed, then we need to optimize the rate in the Gaussian
approximation depending on $\lambda $ and $\alpha _l$. To do so we first ignore log-terms and secondly we
divide the consideration into two cases according to the parameter $m$,
namely Case I: $2/m\ge \lambda $ (moderate values of $m$) and Case II: $2/m\le \lambda $ (large values of $m$).
For  Case I we can find the resulting rate, which is the minimum of 
$\max \{ -\lambda, \lambda +2/m+(a_s-a_l)/2, 3\lambda +6/m+(a_l-1)/2\}$ by balancing the three terms.
This leads to $a_l=1/3+2/3\,a_s-4/(3m)$, $\lambda = 1/12(1-a_s)-4/(3m)$ and the resulting rate
$T^{-r/(6(2r+1))+4/(3m)}$.\\
Exactly along the same lines we obtain for Case II  that  $\lambda = 1/14(1-a_s)-6/(7m)$ and the resulting rate $T^{-r/(7(2r+1))+6/(7m)}$. Both choices are in line with \eqref{eq.all}. \\
Hence, for usual lag-window estimators, i.e. $a _s=1/5$, and
if sufficiently high moments of the time series are assumed to exist, then we can achieve rates up to $ T^{-2/35}$
for the Gaussian approximation.\\
In the limit $a_s\to 0$, we achieve the rate
$T^{-1/14+6/(7m)}$, which reaches, if sufficiently high moments exist, almost  the rate $T^{-1/14}$.
\item[(iii)] If we would instead of sample autocovariances consider sample means of i.i.d. observations, which would make
the small blocks/large blocks considerations in the proof of Theorem 1 superfluous, we could achieve with the presented 
method of proof a Gaussian approximation with rate $T^{-1/8+3/(2m)}$.  
\end{itemize}

\end{remark}

\section{Confidence Bands for the Spectral Density} \label{sec.ConfBand}

To construct confidence bands that appropriately take into account the  local variability of the  lag-window estimator, we make use of the following   Gaussian approximation result  for the   properly standardized lag-window estimators.  

\begin{thm}\label{t.conf-bands}
Suppose that the assumptions of Theorem~\ref{t.asymptotics} are satisfied. Further, assume that
$M_T^3/T \to 0$,
where  additionally to the conditions given in (\ref{eq.a_s}), 
\be\label{eq.thm2}
\lambda+\kappa\leq \min\{1/2-a_s/2-4/m\,,\, 2a_s-4/m\}
\ee
holds true.  
Moreover, we assume for the bias of $ \widehat f_T$ that
\be\label{eq.thm2a}
E \widehat f_T (\lambda)- f(\lambda) = {\mathcal O}(M_T^{-2})
\ee
uniformly in $\lambda $.
Then, 
\begin{align}
		&\sup _{x\in \mathbb R} \Big\vert P\Big( \max _{k=1,\ldots , N_T} 
		\sqrt{\frac{T}{M_T}}  \frac{\vert \widehat f_T(\lambda_{k,T})   - E\widehat f_T(\lambda_{k,T})\vert }{\widehat f_T(\lambda_{k,T})} \le x \Big)-  P\Big(   \max _{k=1,\ldots , N_T} \vert \widetilde\xi_k \vert \le x \Big) \Big\vert \nonumber \\
		  = & \, {\mathcal O}\Big(T^{-\kappa} + T^{-\lambda}\big(\log(N_T)\big)^{3/2}  \Big) ,\nonumber 
	\end{align}
 where $\widetilde\xi _k, k=1, \ldots , N_T$, are jointly normally distributed random variables with zero mean and covariance 
 $E \widetilde \xi_{k_1} \widetilde \xi_{k_2}$ equal to
 \begin{align}\label{eq.Ck1k2}
 C_T(k_1,k_2)= \frac{1}{f(\lambda_{k_1,T})f(\lambda_{k_2,T}) }  & \sum _{j_1,j_2=0}^{M_T} a_{k_1,j_1} a_{k_2,j_2} \nonumber  \\ 
 \times  \frac{1}{T}\sum _{t=j_1+1}^T  \sum _{s=j_2+1}^T 
 & E\Big[\big( X_t X_{t-j_1} -\gamma (j_1)\big) 
 	 \big( X_s X_{s-j_2} -\gamma (j_2)\big)\Big].  
 \end{align}
\end{thm}
\medskip

\begin{remark}
{~}
\begin{itemize}
\item[(i)] Choosing $a_s$ MSE optimal, that is $a_s=1/5$,  the corresponding optimal $\lambda=\kappa= 2/35$ (see Remark~\ref{r.1}) satisfy \eqref{eq.thm2}.
\item[(ii)] Under the additional assumption  $\lim _{u\to 0} (1-w(u))/u^2 >0$ for the lag-window, validity of \eqref{eq.thm2a}
is shown in \citeasnoun[Theorem 9.3.3]{A71}.\\
For the case $M_T \sim T^{1/5}$ \citeasnoun{BergPolitis2009} have shown in their Theorem~1 that, under our assumptions,
for  lag-window estimators with a so-called flat-top lag-window $w$ even
$$
\sup_{\lambda} \vert E \widehat f_T (\lambda)- f(\lambda)\vert = o(M_T^{-2})
$$
holds. This implies that the results of Theorem \ref{t.conf-bands} also hold for the important class of flat-top
lag-window estimators introduced in \citeasnoun{PolitisRomano87}.
It is worth mentioning that flat-top lag windows fail to fulfill $\lim _{u\to 0} (1-w(u))/u^2 >0$.
\end{itemize}
\end{remark}
\medskip

\noindent Theorem~\ref{t.conf-bands} motivates  the following multiplier bootstrap procedure to construct   a
 confidence  band for the smoothed spectral density 
$$\widetilde{f}_T(\cdot):=E(\widehat{f}_T(\cdot)) .$$ 

\begin{description}
\item[Step 1.]  For $k_1,k_2 \in \{1,2, \ldots, N_T\}$, let  $ \widehat{C}_T(k_1,k_2)$  be an estimator of the covariance  $C_T(k_1,k_2)$ given in (\ref{eq.Ck1k2}) which ensures that the $N_T\times N_T$ matrix $ \widehat{\Sigma}_{N_T}$  with the $ (i,j)$-th element  equal to   $ \widehat{C}_T(i,j)$ is non-negative definite. 
\item[Step 2.] Generate  random variables $ \xi^\ast_1,\xi^\ast_2, \ldots, \xi^\ast_{N_T}$, where 
\[ \big( \xi^\ast_1,\xi^\ast_2, \ldots, \xi^\ast_{N_T}\big)^\top \sim {\mathcal N}\Big(0_{N_T}, \widehat{\Sigma}_{N_T} \Big),\]
with $ 0_{N_T}$ a $N_T$-dimensional vector or zeros  and covariance matrix $ \widehat{\Sigma}_{N_T} $.
Let 
$\xi^\ast_{\max} =\max\big\{|\xi^\ast_1|,|\xi^\ast_2|, \ldots, |\xi^\ast_{N_T}| \big\}$ and for  $ \alpha \in (0,1)$ given, denote by 
$ q^\ast_{1-\alpha}$ the  upper $(1-\alpha)$ percentage point of the distribution  of $\xi^\ast_{\max}$.
 \item[Step 3.]  A simultaneous $(1-\alpha)$-confidence band for $\widetilde{f}(\lambda_{j,T})$, $j=1,2, \ldots, N_T$, is then given by 
\begin{equation} \label{eq.CBand1}
 \Big\{  \Big[\widehat{f}(\lambda_{j,T}) \Big(1-q^\ast_{1-\alpha}\sqrt{\frac{M_T}{T}}\Big),  \ \  \widehat{f}(\lambda_{j,T}) \Big(1+q^\ast_{1-\alpha}\sqrt{\frac{M_T}{T}}\Big) \Big], \  j=1,2, \ldots , N_T.\Big\}
 \end{equation}
\end{description}

\vspace*{0.2cm}

Notice that the distribution of $ \xi^\ast_{\max} $ in Step 2 can be 
estimated
via Monte-Carlo simulation.  
While $ f(\lambda_{k_i,T})$, $i=1,2$,  appearing in the expression for $ C_T(k_1,k_2)$ can be replaced by the estimator $ \widehat{f}_T(\lambda_{k_i,T})$, the important step in the bootstrap algorithm proposed, is  the estimation of the  covariances 
$$ \sigma_T(j_1,j_2)= \frac{1}{T}\sum _{t=j_1+1}^T  \sum _{s=j_2+1}^T  E\Big[\big( X_t X_{t-j_1} -\gamma (j_1)\big) 
 	 \big( X_s X_{s-j_2} -\gamma (j_2)\big)\Big]  $$
	 appearing in $ C_T(k_1,k_2)$.  To obtain an estimator for this quantity which has  some  desired consistency properties,  we follow \citeasnoun{Zhang_etal2022}. In particular,  let $ K$ be  a kernel function satisfying the following conditions.
\vspace*{0.2cm}

\noindent {\bf Assumption 4:} $ K:\R \to [0,+\infty)$  is symmetric, continuously differentiable and decreasing on $[0,+\infty)$  with $K(0) =1$  and $ \int K(x)\,dx<\infty$. The Fourier transform  of $K$ is integrable and nonnegative on $\R$. 
\vspace*{0.2cm}

\noindent Consider  next the  estimator   
	 \begin{equation} \label{eq.Est-c}
	 \widehat{\sigma}_T(j_1,j_2) =  \frac{1}{T}\sum _{t=j_1+1}^T  \sum _{s=j_2+1}^T K\Big(\frac{t-s}{b_T}\Big)\big( X_t X_{t-j_1} -\widehat{\gamma} (j_1)\big) 
 	 \big( X_s X_{s-j_2} -\widehat{\gamma} (j_2)\big),
	 \end{equation} 
of $ \sigma_T(j_1,j_2)$, where $ b_T >0$ is a bandwidth parameter satisfying 
$b_T\rightarrow \infty$ as $ T\rightarrow\infty$. 
Note that the conditions on the Fourier transform of $K$ in Assumption~4 guarantee positive semi-definiteness  of $\widehat\sigma_T(j_1,j_2)$. This  assumption is satisfied if $K$ is, for instance, the  Gaussian kernel.  

\noindent The following consistency result  for the estimator  proposed in (\ref{eq.Est-c}) can  be established.  
  
\begin{proposition}\label{l.sigma-hat}
	Suppose that Assumption~1  and Assumption~4 hold true.  Let $ b_T\sim T^c$ with 
	$ c < 1/2 - 8a_s/m$, where $a_s$
	is as in Theorem~\ref{t.conf-bands}. 
	Then,
\[ 
\sup_{1\leq j_1,j_2\leq M_T} \big |\widehat{\sigma}_T(j_1,j_2) - \sigma_T(j_1,j_2)\big| =
O_P\Big(
\frac{1}{b_T} + \frac{b_T M_T^{8/m}}{\sqrt{T}}\Big).
\]
\end{proposition}
\vspace*{0.3cm}

\noindent Let, as usual, $P^*$ denote the conditional probability given the time series $X_1,\dots, $ $X_T$.  We then have the following result which proves consistency of the multiplier bootstrap procedure proposed.

\begin{thm} \label{th.MPBoot}
Suppose that the  conditions  of Theorem~\ref{t.conf-bands} hold true and that 
$ b_T\sim T^c$, where 
\[ 
0 <
a_s< c < \frac{1}{2} -
a_s \big(1 + \frac{8}{m}\big).
\]
Let    $ \xi_k^\ast$, $k=1,2, \ldots, N_T$, be Gaussian random variables   generated   as in Step 2 of the multiplier bootstrap algorithm with $\widehat{\Sigma}_{N_T}$ the $N_T\times N_T$ matrix  the $ (k_1,k_2)$-th element of which equals  
 \begin{align}\label{eq.HatCk1k2}
 \widehat{C}_T(k_1,k_2)= \frac{1}{\widehat{f}(\lambda_{k_1,T})\widehat{f}(\lambda_{k_2,T}) }  & \sum _{j_1,j_2=0}^{M_T} a_{k_1,j_1} a_{k_2,j_2}  \widehat{\sigma}_T(j_1,j_2),
 \end{align}	
and $  \widehat{\sigma}_T(j_1,j_2)$ given in (\ref{eq.Est-c}).
Then, as $ n \rightarrow \infty$,
\begin{align} \label{eq.MPBootBound} 
 \sup _{x\in \mathbb R} \Big\vert P\big( \max _{k=1,\ldots , N_T} &
		\sqrt{\frac{T}{M_T}}  \frac{\vert \widehat f_T(\lambda_{k,T})    - E\widehat f_T(\lambda_{k,T})\vert }{\widehat f_T(\lambda_{k,T})} \le x \big) 
			 -  P^\ast\big(   \max _{k=1,\ldots , N_T} \vert \xi^\ast_k \vert \le x \big) \Big\vert  \nonumber \\
			 = &\,  {\mathcal O}\Big(T^{-\kappa} + T^{-\lambda}\big(\log(N_T)\big)^{3/2}  \Big) \nonumber \\
			 &  + o_P\big(\{\sup_{k=1,\ldots , N_T} |\widehat{f}_T(\lambda_{k,T})-f(\lambda_{k,T})|\}^{1/6} + T^{-\rho/6}\big)
\end{align}
for $ \rho =\min\{c-
a_s ,1/2-c-
a_s (1+8/m)\}$. 			 
\end{thm}

\begin{remark}
{~}
\begin{itemize}
\item[(i)] So far we have considered the construction of  (simultaneous) confidence bands for  the (smoothed) spectral density $ \widetilde{f}(\lambda_{j,T}) =\E(\widehat{f}_T(\lambda_{j,T})) $ over the
Fourier frequencies  $\lambda_{j,T}$, $ j=1,2, \ldots, N_T$. To extend  the procedure proposed to  one that also delivers  an asymptotically valid  simultaneous confidence band for the spectral density $ f$ itself,  a (uniformly) consistent estimator of the (rescaled)  bias term $ B_T(\lambda_{j,T})=\sqrt{T/M_T}\big(\E(\widehat{f}_T(\lambda_{j,T}) ) -f(\lambda_{j,T})\big)/\widehat{f}_T(\lambda_{j,T})$ for $ j=1,2, \ldots, N_T$, is needed, provided  $ B_T(\lambda_{j,T})$ does not vanish asymptotically; see (\ref{spec_dens_normal}).   If fact, it can be easily seen that a (theoretically) valid confidence band for $ f$ which takes into account   the bias in estimating $f$,  is  given by 
\begin{align} \label{eq.CBand2}
 \Big\{  \Big[\widehat{f}(\lambda_{j,T}) & \Big(1-\sqrt{\frac{M_T}{T}}\big(q^\ast_{1-\alpha}+B_T(\lambda_{j,T}\big)\Big),   \\
 &  \ \  \widehat{f}(\lambda_{j,T}) \Big(1+\sqrt{\frac{M_T}{T}}\big(q^\ast_{1-\alpha}-B_T(\lambda_{j,T}\big)\Big) \Big],  \nonumber   j=1,2, \ldots , N_T\Big\};
 \end{align}
compare to  (\ref{eq.CBand1}). As in many other nonparametric inference problems  too, different  approaches can be considered for this purpose; see 
\citeasnoun{Calonicoetal2018}  
and the references therein for  the cases of nonparametric density and  regression estimation in an i.i.d. set up.   One approach is an explicit bias  correction  that  uses  a plug-in type estimator   of $B_T(\lambda_{j,T})$ based on  the fact that, under certain conditions,  see (\ref{spec_dens_normal}),
$ \sqrt{T/M_T}\big(\E(\widehat{f}_T(\lambda_{j,T}) ) -f(\lambda_{j,T})\big) = 
C W f^{\prime \prime}(\lambda_{j,T}) + o(1)$. 
This   approach requires a consistent (nonparametric) estimator of the second order derivative 
$f^{\prime \prime}$ of the spectral density. A different 
approach proposed in the literature is the so-called 'undersmoothing'.
The idea here  is to make  the bias term $ \sqrt{T/M_T}\big(\E(\widehat{f}_T(\lambda_{j,T}) ) -f(\lambda_{j,T})\big)$  asymptotically negligible. This can be achieved 
  by using  a  truncation lag $M_T$ which  increases  to infinity at a rate  faster than 
the (MSE optimal) rate $ T^{1/5}$.   An alternative approach  is to perform a bias correction by using flat-top kernels. See \citeasnoun{Politis24} for a recent discussion, where also an alternative approach to variance stabilization, respectively, studentization, called the confidence region method,  has been proposed.
Despite the fact that the problem  of properly  incorporating  the (possible)  bias in the construction of (simultaneous) confidence bands for the spectral density $f$ is an interesting one, we do not further pursue this problem in this  paper.
\item[(ii)] Equation \eqref{eq.MPBootBound} contains as part of the convergence rate for the multiplier bootstrap 
the sup-distance of the lag-window estimator to the spectral density over all Fourier frequencies $\lambda _{k,T}$. 
Assuming a geometric rate for the decay of the physical dependence coefficients $\delta_m(k)$, one obtains from
\citeasnoun{WuZaffaroni2018}, Theorem 6, the uniform convergence rate $(M_T \log (M_T)/T)^{1/2}$ over all 
frequencies for lag-window estimators, which dominates the uniform term over all Fourier frequencies 
in \eqref{eq.MPBootBound} and usually converges to zero faster
than the other parts of the bound therein.
\end{itemize}
\end{remark}

\section{Simulations} \label{sec.Sim}

In this section we investigate by means of simulations,    the finite sample performance of the Gaussian approximation  and of the corresponding multiplier bootstrap  procedure proposed for the 
construction of simultaneous confidence bands.   For this purpose, time series  of length $T=  256,\; 512$ and  $1024$  have  been generated from the following three  time series models:
\begin{enumerate}
\item[] \hspace*{-1cm} {\bf Model I:}  \ $X_t=0.8 X_{t-1} + \varepsilon_t$,
\item[]  \hspace*{-1cm} {\bf Model II:} \ $X_t = 1.3X_{t-1} -0.75 X_{t-2} + u_t$ with $ u_t =\varepsilon_t \sqrt{1+0.25 u^2_{t-1}} $,
\item[] \hspace*{-1cm} {\bf Model III:} \  $ X_t=\big(0.4+0.1\varepsilon_{t-1}\big)X_{t-1} + \varepsilon_{t}$.
\end{enumerate}
In all models  the innovations  $\varepsilon_t$ are  chosen to be i.i.d.  standard Gaussian.
The empirical coverage over  $R=500$ repetitions of each model has been calculated for two different nominal  coverages,  $ 90\%$ and $ 95\%$. The  lag-window estimator of the  spectral density  used  has been obtained using   the Parzen lag-window 
and   truncation lag $ M_T$.  The Gauss kernel with different values of the parameter $ b_T$ has been used for obtaining the covariance estimators  $   \widehat{\sigma}_T(j_1,j_2)$ given in~(\ref{eq.Est-c})  and  which are used  for the calculation of the   covariance  matrix
$\widehat{C}_T(k_1,k_2)$;  see equation (\ref{eq.HatCk1k2}).  All bootstrap approximations are based on $B=1,000$ replications. Table 1  presents  empirical coverages as well as  mean lengths of the confidence bands obtained, where the mean lengths are  calculated as 
\[ML:=  2  \sqrt{\frac{M_T}{T}} \frac{1}{N_T} \sum_{j=1}^{N_T}\widehat{f}_T(\lambda_{j,T})\frac{1}{R} \sum_{\ell=1}^{R} q_{1-\alpha}^{\ast,(\ell)},\]
 with $ R$  the number of repetitions and   $ q^{\ast, (\ell)}_{1-\alpha}$ the  upper $(1-\alpha)$ percentage point of the distribution  of $\xi^\ast_{\max}$ obtained in the $ \ell$-th repetition; also see 
 expression  (\ref{eq.CBand1}).   Figure~\ref{fig.AR-ARCH} shows  averaged 90\% and 95\%  confidence bands  obtained using the method proposed in this paper for sample sizes of $T=512$ and $T=1024$ and for Model II.

\begin{center}
\begin{small}
\vspace*{0.3cm}
\begin{tabular}{|lll|ccc|ccc|ccc|}
\hline
  &    & &  \multicolumn{3}{c|}{\bf Model I} &  \multicolumn{3}{c|}{\bf Model II} & \multicolumn{3}{c|}{\bf Model III}   \\
 \hline 
{\bf T=256}    &  & & &  & &  &  &   & &    &  \\
$M_T=10$&  \multicolumn{2}{r|}{ $b_T=$ }  & 1.0& 1.5 &2.0 & 6.5 & 7.0  & 7.5 &  1.0 & 1.5 & 2.0  \\
           &  90\% &Cov & 92.6&  90.6 &  89.2 & 90.8 &  90.0   &  90.0 & 86.0 & 82.8  & 82.2 \\
            &   &ML & 0.56&  0.47&  0.43 & 1.46  & 1.42   &  1.38 & 0.18 & 0.17 & 0.17\\
         &  95\% &Cov & 94.0&  92.4 &  91.4 & 94.4 & 94.0   &  92.8 & 91.0 & 89.4 & 87.6 \\
            &   &ML & 0.63&  0.53 &  0.49 & 1.68  & 1.64   &  1.59 & 0.20 & 0.19 & 0.19 \\
{\bf T=512}    &  & & &  & &  &  &   & &    &  \\
$M_T=14$&  \multicolumn{2}{r|}{ $b_T=$ } &  1.5& 2.0&2.5 & 9.5 & 10.0  & 10.5  &  1.0&  1.5 &  2.0  \\
            &  90\% &Cov & 91.6&  88.8 &  87.6 & 90.4 & 90.2  &  89.8    & 87.8 &  84.2 & 83.0  \\
            &   &ML &   0.41 &  0.37 & 0.36 & 1.05 & 1.03& 1.02     &  0.15& 0.14 & 0.15 \\
         &  95\% &Cov & 94.6&  91.8 &  91.2 &  93.8& 94.0& 93.2    & 92.2&  89.8 & 88.0  \\
            &   &ML & 0.46&  0.42 &  0.40& 1.19  &  1.17 & 1.16   & 0.17 & 0.16 & 0.16 \\
 {\bf T=1024}    &  & & &  & &  &  &   & &    &  \\
 $M_T=18$&  \multicolumn{2}{r|}{ $b_T=$ }  & 2.0& 2.5 & 3.0 & 11.0  & 11.5  & 12.0 & 1.0 & 1.5  & 2.0  \\
           &  90\% &Cov & 92.8  & 90.8   & 88.6    & 90.0   & 89.4   &  89.2  & 89.0  &86.4  & 86.0   \\
            &   &ML & 0.31 &  0.29 &  0.28  & 0.79  &  0.78   &   0.77  & 0.12  &0.12  & 0.12  \\
                   &  95\% &Cov & 96.0 &  94.8    &  94.0  &  94.0 & 94.2   &  93.8 & 94.6     &93.2   & 92.0   \\
            &   &ML & 0.34 &  0.32 &  0.31  & 0.88   &  0.88   &   0.87   & 0.14   & 0.13 & 0.13  \\
\hline                
\end{tabular} 
\end{small}
\end{center}
\vspace*{0.05cm}
 \begin{center}
    \begin{minipage}{10.8cm}
{\bf Table 1.} Empirical coverages (Cov) and mean lengths (ML) of simultaneous confidence bands  (\ref{eq.CBand1})  for different sample sizes and different values of the parameters $M_T$ and $b_T$.
    \end{minipage}
    \end{center}

\vspace*{0.5cm}

\begin{figure}[htbp]
    \begin{center}
         \includegraphics[height=6.7cm,width=10cm]{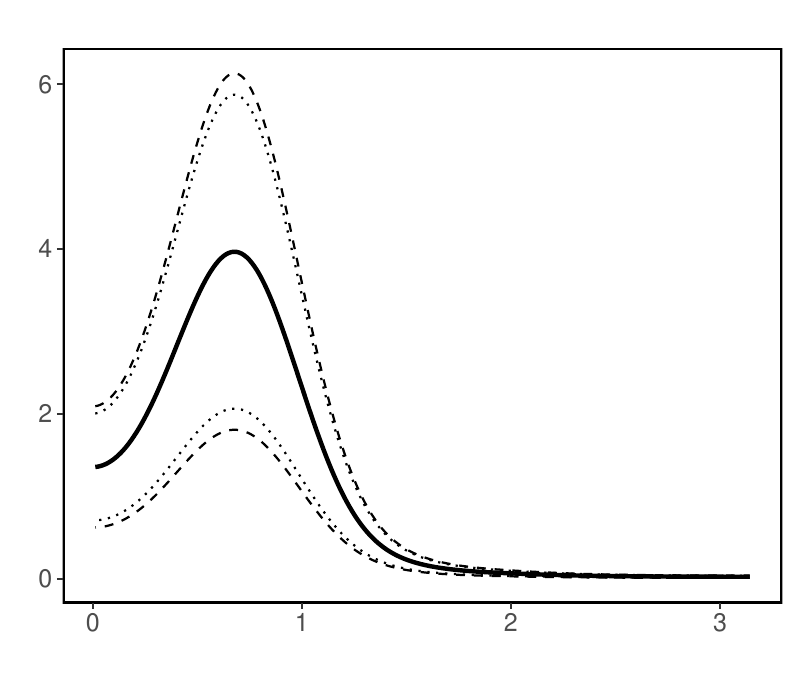} \\  
        \includegraphics[height=6.7cm,width=10cm]{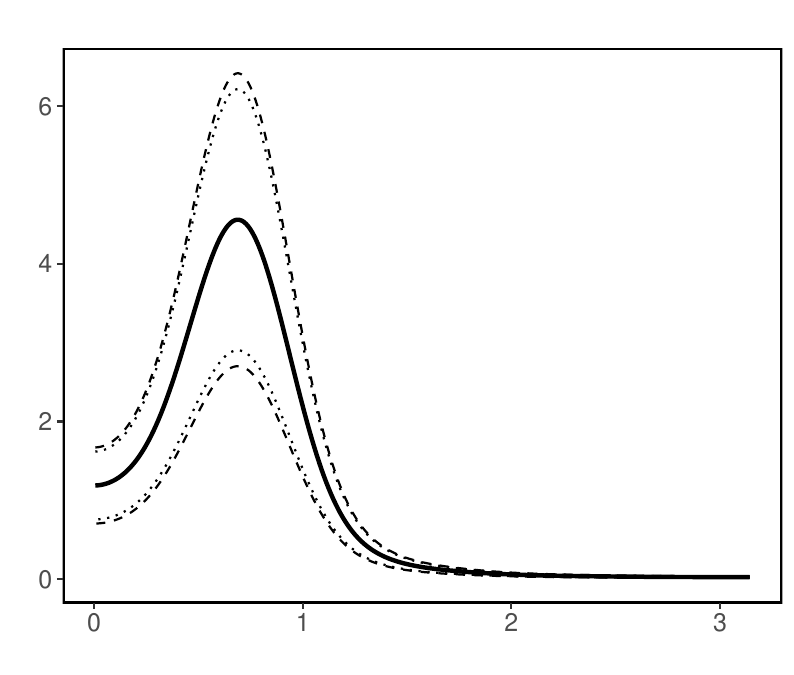} 
      \end{center}
    \vspace{-.4cm}
    \begin{center}
    \begin{minipage}{11cm}
    \vspace*{0cm}
    \caption{Plot of  $\widetilde{f}(\lambda_{j,T}) $ (solid line)  together with   $90\%$  and $95\%$  averaged confidence bands  (dotted and dashed lines, respectively), 
     for  time series of length $T=512$ and $ b_T=10$  (top) and $T=1024$  and $b_T=11.5$ (bottom)   stemming from Model II. \label{fig.AR-ARCH}}
    \end{minipage}
    \end{center}
\vspace{-0.3cm}
\end{figure}

We also compare the performance of the approach proposed in this paper with that based on a Gumbel-type approximation of the  maximum deviation of  the  centered lag-window spectral density estimator  evaluated over  a much coarser grid of frequencies in the interval $(0,\pi]$. To elaborate and   based on Theorems 3-5  of  \citeasnoun{LiuWu2010} derived under different conditions,  an asymptotically $(1-\alpha)$ simultaneous confidence band for  $ \E(\widehat{f}_T(\lambda_s))$ over the   set of frequencies $ \lambda_s= s\pi/M_T$ for $ s=1,2, \ldots, M_T$,  is given by 
\begin{align} \label{eq.ConfBand-Gumbel}
\Big\{ \Big[ \widehat{f}_T(\lambda_s)- \widehat{C}_{\alpha,T}  , \widehat{f}_T(\lambda_s)+\widehat{C}_{\alpha,T}\Big], \ s=1,2, \ldots, M_T\Big\},
\end{align}
where 
\begin{equation} \label{eq.ConfBand-Gumbel2} 
 \widehat{C}_{\alpha,T}=\sqrt{\frac{M_T}{T}\big(c_{1-\alpha} +\mu_T \big) \widehat{f}^2_T(\lambda_s) W_2}, \ \ \mu_T=2\log(M_T)-\log(\pi \log(M_T)) 
 \end{equation}  
and $ W_2=\int_{-1}^1 w^2(u)du = 151/280$  in  the case of the Parzen lag-window.  Furthermore, $ c_{1-\alpha}$ denotes  the  $(1-\alpha)$ percentage point of the standard Gumbel distribution.
Table 2 summarizes  empirical coverages and mean lengths of the confidence bands   (\ref{eq.ConfBand-Gumbel})-(\ref{eq.ConfBand-Gumbel2})  obtained  over $ R=500$ repetitions  for the same models and sample sizes considered in Table 1.  Additionally and 
in order to see the effect of dependence of the time series at hand,   we  also  report results for 
 the case of an  i.i.d. process  with  $ X_t \sim {\mathcal N}(0,1)$. Since  the set of frequencies $\lambda_s$ captured  by  the   confidence band   (\ref{eq.ConfBand-Gumbel})-(\ref{eq.ConfBand-Gumbel2})  solely depends on   the truncation lag $M_T$,  we present results  for   different values of  this parameter.

\begin{center}
\begin{small}
\begin{tabular}{|llc|cc|cc|cc|cc|}
\hline
  &    &  & \multicolumn{2}{c|}{\bf i.i.d.} &  \multicolumn{2}{c|}{\bf Model I} & \multicolumn{2}{c|}{\bf Model II} & \multicolumn{2}{c|}{\bf Model III}  \\ 
  &    &  & 90\% &  95\% & 90\% & 95\%  & 90\% & 95\%  & 90\% & 95\%   \\
  \hline 
  {\bf  T=256} &  &  &  &  &  &   & &   &  &  \\
   & $M_T=10$ & Cov & 84.0& 89.4& 76.2& 84.0 &68.4& 75.8 & 60.2& 65.8 \\
      &  & ML  & 0.12&  0.13& 0.25& 0.28& 0.77& 0.84 & 0.16& 0.18 \\
      & $M_T=14$ & Cov & 79.0& 85.0& 74.8& 80.8 &71.2 & 76.6 & 59.6& 66.0 \\
      &  & ML  & 0.15&  0.16 & 0.32& 0.35& 0.98& 1.07& 0.21 & 0.22  \\
      & $M_T=22$ & Cov & 69.0 & 74.4& 68.2 &74.8 & 65.6 & 71.4 & 55.8& 64.0 \\
      &  & ML  & 0.20&  0.21 & 0.45& 0.49& 1.32& 1.43& 0.28 & 0.30  \\
   {\bf  T=512} &  &  &  &  &  &   & &   &  &  \\
    & $M_T=14$ & Cov & 84.0& 87.8& 81.6& 85.2 & 75.4& 80.0 & 59.0 & 63.8 \\
      &  & ML  & 0.11&  0.12 & 0.23& 0.26& 0.69& 0.76& 0.15 & 0.16  \\
      & $M_T=18$ & Cov & 82.8& 86.6& 82.4& 85.6 & 74.2& 80.2 & 59.4 & 64.6 \\
      &  & ML  & 0.12&  0.14& 0.28& 0.30& 0.82& 0.89 & 0.17& 0.19 \\
      & $M_T=26$ & Cov & 78.6& 84.0& 78.2 &84.0.8 & 71.2 & 72.2 & 58.4& 66.2 \\
      &  & ML  & 0.15&  0.17 & 0.37& 0.39& 1.03& 1.12& 0.22 & 0.24  \\
  {\bf  T=1024} &  &  &  &  &  &   & &   &  &  \\
   & $M_T=18$ & Cov & 88.0& 91.0 & 85.6& 88.6 &76.0& 82.4 & 59.6& 66.2 \\
      &  & ML  & 0.09  & 0.10 & 0.20& 0.22& 0.58& 0.63 & 0.12& 0.13 \\
      & $M_T=22$ & Cov & 85.2 & 90.6 & 83.6& 80.8 &76.2 & 82.6 & 60.2& 67.8 \\
      &  & ML  & 0.10&  0.11 & 0.23 & 0.25& 0.66& 0.72& 0.14 & 0.15  \\
      & $M_T=30$ & Cov & 83.0 & 88.0& 82.2 &86.6 & 75.2 & 81.2 & 62.4& 68.4 \\
      &  & ML  & 0.12&  0.13 & 0.29& 0.31& 0.81& 0.87& 0.17 & 0.18  \\
 \hline                 
\end{tabular} 
\end{small}
\end{center}
\vspace*{0.05cm}
 \begin{center}
    \begin{minipage}{10.8cm}
{\bf Table 2.} Empirical coverages (Cov) and mean lengths (ML) of the simultaneous confidence bands  (\ref{eq.ConfBand-Gumbel})  for different sample sizes and different values of the  truncation-lag  $M_T$.
    \end{minipage}
    \end{center}
    
\vspace*{0.2cm}

\noindent
As it  can be seen from Table~1, the empirical coverages  of our   confidence bands  are,  in general, close to  the desired levels and  they  improve  as the sample size increases. While the  choice of the  parameter  $M_T$, which specifies  the number of empirical autocovariances effectively used  in obtaining $\widehat f_T$,   seems  not  to be   important for  the performance of the method based on Gaussian approximation, this method   is   more sensitive with respect to the choice of the bandwidth parameter $b_T$ used  in the estimation of the covariance matrix of the approximating Gaussian variables. This parameter should be chosen larger for process with a  stronger dependence structure compared to rather weakly dependent data. This is a standard observation in the context of covariance estimation, see e.g.~\citeasnoun{A91}.  Also, differences of the empirical coverages between the  different models can be seen, where  the bilinear Model III  seems to be a rather  difficult case. This model  clearly needs larger sample sizes than the other two models considered in order to obtain coverages which are close to the nominal ones.   An inspection of Table 2 shows that the method based on   the asymptotic Gumbel approximation   and which uses  a much coarser grid of frequencies,  has  difficulties in achieving the desired confidence levels even in the most simple case of a Gaussian i.i.d. process while for Model III this method leads to  quite low empirical coverages even for $T=1024$.  Despite the fact that,   overall,  the  empirical coverages for the i.i.d. case, Model I and Model II,  improve  slowly as the sample size increases, the results  obtained  heavily depend  on the choice of  the truncation lag  $M_T$  and the coverages  achieved  stay  in most cases quite  below the desired level even for the largest sample size used in the simulation study. 

\section{Auxiliary Lemmas and Proofs}

Throughout this section, $C$ denotes a generic constant that may vary from line to line. We first state the following useful lemmas. See also \citeasnoun{XiaoWu2014} for related results to Lemma \ref{l.cov}.

\begin{lemma}\label{l.cov}
Under Assumption 1,
the following assertions hold true:
\begin{enumerate}
\item[(i)] \ $ |\gamma(j)| \leq C(1+|j|)^{-\alpha}$,
\item[(ii)] \
$\sum_{j\in\Z}|j|^r|\gamma(j)|<\infty~\forall r<\alpha-1$,
\item[(iii)] \  $\max_{j_1,j_2\geq 0} \vert E(Y_{t,j_1}Y_{s,j_2}) \vert \leq C(1+|s-t|)^{-\alpha}$, where $ Y_{t,j}=X_{t}X_{t+j}-\gamma (j)$.
\end{enumerate}
\end{lemma}
\begin{proof}
	\begin{itemize}{}
		\item[(i)] 
We make use of the bound
\begin{equation}\label{B-13a}
\Vert P_0(X_s)\Vert _m \le \delta _m(s), s \ge 0,
\end{equation}
of the so-called projection operator $P_{j-s}(X_j) := E[X_j\vert {\cal F}_{j-s}] - E[X_j\vert {\cal F}_{j-s-1}]$ for 
$s\ge 0$ and $j\in \mathbb Z$, where
${\cal F}_i:= \sigma (e_i, e_{i-1}, \ldots )$. For a proof of \eqref{B-13a} we refer to  \citeasnoun{Wu2005}, Theorem 1.\\
Note that $X_t=\sum _{s=0}^{\infty} P_{t-s}(X_t)$ a.s. and in $L_1$. 
This is the case because $E[X_t\vert {\cal F}_{t-s}]$ converges for $s\to \infty$ by the backward martingale 
convergence theorem 
a.s. and in
$L_1$ towards a limit measurable with respect to ${\cal F}_{-\infty}$, which is trivial because of the i.i.d. structure of
$(e_i)$. Therefore the limit is constant and coincides with the mean of $X_t$ which is assumed to be zero.
Then (i) follows since,
\begin{align*}
\vert \gamma (j)\vert =& \vert E(X_0\, X_j)\vert = \Big\vert \sum _{s_1,s_2=0}^{\infty} E(P_{-s_1}(X_0)P_{j-s_2}(X_j))\Big\vert \\
=& \Big\vert \sum _{s=0}^{\infty} E(P_{-s}(X_0)P_{-s}(X_j))\Big\vert , \ \text{because }
E(P_{-s_1}(X_0)P_{j-s_2}(X_j))=0, -s_1\neq j-s_2      \\
\le & \sum _{s=0}^{\infty} E\vert P_{0}(X_s)P_{0}(X_{j+s})\vert, \ \text{because of stationarity} \\
\le & \sum _{s=0}^{\infty} \Vert P_{0}(X_s)\Vert _{m^{\prime}} \Vert P_{0}(X_{j+s})\Vert _m ,\ \text{where } 1/m+1/m^{\prime}=1\\
\le & \sum _{s=0}^{\infty} \Vert P_{0}(X_s)\Vert _{m} \Vert P_{0}(X_{j+s})\Vert _m , \ \text{because } m^{\prime }\le m 
\text{ for } m\ge 2 \\
\le & \sum _{s=0}^{\infty} \delta _m(s) \delta _m(j+s), \text{ by } \eqref{B-13a}.
\end{align*}
From this bound we get for $ j\geq 0$, 
\begin{align*}
|\gamma(j)| \leq C \sum_{s=0}^{\infty} \frac{1}{(1+s)^{\alpha} (1+s+j)^{\alpha}} \leq C \frac{1}{(1+j)^{\alpha}} \sum_{s=0}^{\infty} \frac{1}{(1+s)^{\alpha} } \leq C (1+j)^{-\alpha} .
\end{align*}
\item[(ii)] The assertion follows because  
\[ \sum _{j=1}^{\infty } j^r \vert \gamma (j)\vert \leq C \sum _{j=1}^{\infty } j^r(1+j)^{-\alpha} \le  C\sum_{j=1}^\infty \frac{1}{(1+j)^{\alpha-r}} <\infty \]
for $ r<\alpha-1$.
\item[(iii)] We assume without loss of generality that $t>s$ and define   $X_r^{(k)}\,:=\,E[X_r\mid  \mathcal F_{r,k}]$. 
The following three cases can then occur.   
If  $ j_1\leq t-s$, then
\begin{align*}
\vert E(Y_{t,j_1}Y_{s,j_2})\vert & \leq \vert E(X_tX_{t-{j_1}}X_sX_{s-j_2}| + |\gamma(j_1)\gamma(j_2)|\\
& = \vert E(\big(X_t - X_t^{(t-s-1)})X_{t-j_1}X_s X_{s-j_2}\big)\vert +  |\gamma(j_1)\gamma(j_2)| \\
& \leq \|X_{t-j_1}X_sX_{s-j_2}\|_{m/(m-1)} \|X_t - X_{t}^{(t-s-1)}\|_m + C(1+j_1)^{-\alpha}\\
& \leq C\,(1+t-s)^{-\alpha}.
\end{align*}
If $j_1 \in [(t-s)/2,t-s]$, then
\begin{align*}
\vert E(Y_{t,j_1}Y_{s,j_2})\vert & \leq \vert E(X_tX_{t-{j_1}}X_sX_{s-j_2}| + |\gamma(j_1)\gamma(j_2)|\\
& = \vert E(\big(X_t - X_t^{(t-j_1-1)})X_{t-j_1}X_s X_{s-j_2}\big)\vert +  |\gamma(j_1)\gamma(j_2)| \\
& \leq \|X_{t-j_1}X_sX_{s-j_2}\|_{m/(m-1)} \|X_t - X_{t}^{(t-j_1-1)}\|_m + C(1+j_1)^{-\alpha}\\
& \leq C(1+j_1)^{-\alpha} \leq C(1+(t-s)/2)^{-\alpha} \leq C 2^\alpha (1+t-s)^{-\alpha}.
\end{align*}
Finally, if $ j_1\in[0,(t-s)/2]$, then
\begin{align*}
\vert E(Y_{t,j_1}Y_{s,j_2})\vert & \leq \vert E((Y_{t,j_1}-Y_{t,j_1}^{(t-s-1)})Y_{s,j_2}) \vert \\
& \leq  \|Y_{t,j_1}-Y_{t,j_1}^{(t-s-1)}\|_{m/(m-2)}\| Y_{s,j_2} \|_{m/2}\\
& \leq C\sum_{k=t-s}^\infty \delta_m(k) +C \sum_{k=t-s-j_1}^\infty \delta_m(k) \\
& \leq C(1+t-s)^{-\alpha} + C\sum_{k=(t-s)/2}^\infty \delta_m(k) \\
& \leq C \,2^{\alpha} (1+t-s)^{-\alpha}.
\end{align*}
\end{itemize}

\end{proof}
\medskip

\begin{lemma} \label{l.inequalities}
	Suppose that  Assumptions 1 to 3 hold. Then, we have
\bea \label{B-13}
\big\Vert \frac{1}{\sqrt{T}} \sum _{j=0}^{M_T} \sum _{t=j+1}^T  a_{k,j} \big( X_t X_{t-j} -\gamma (j)\big) \big\Vert _{m/2}
\le \mbox{C}
\eea
and
\bea\label{eq.var-lower}
&&\hspace{-0.7cm}
\sup_{k\in\{1,\dots, N_T \}} \Big|\big\Vert \frac{1}{\sqrt{T}}  \sum _{t=1}^T Z_{t,k} \big\Vert_2^2 -  \,\frac{f^2(\lambda _{k,T})}{  M_T} 	\sum_{j=-M_T }^{M_T } 
w^2\big(\frac{j}{M_T}\big) \,(1+\cos(2j\,\lambda_{k,T}))  \Big|\nonumber\\
&&=o(1).
\eea
Further
\bea \label{B-15a}
\Big\|\frac{1}{\sqrt T}\sum_{t=1}^T (Z_{t,k}-\widetilde Z_{t,k}^{(s)})\Big\|_{m/2} \,\leq\,  \mbox{C}\cdot d_{s,m} 
\eea
with the $m$-dependent random variables ($m=2s$)  $\widetilde Z_{t,k}^{(s)}$, $ k=1,2, \ldots, N_T,$
defined  as:
\begin{equation} \label{eq.Z-tilde}
\widetilde Z_{t,k}^{(s)} \,:=\, \sum_{j=0}^{M_T}a_{k,j}\, \left(X_t^{(s)}X_{t-j}^{(s)}\, -\, E[X_t^{(s)}X_{t-j}^{(s)}]\right)\, \ind_{t>j}\, ,
\end{equation}
where $X_r^{(s)}\,:=\,E[X_r\mid  \mathcal F_{r,s}]$ and $s\geq M_T$, and with
\begin{equation} \label{eq.dsm}
d_{s, m}\,:=\,\sum_{h=0}^\infty\min\Big\{\delta_m(h)\,, \,\big(\sum_{j=s+1}^\infty\delta_m^2(j)\big)^{1/2}\Big\}.
\end{equation}
\end{lemma}
\medskip

\begin{proof}
Inequality \eqref{B-13} is a direct consequence of (S8) from the supplement of \citeasnoun{XiaoWu2014} since 
\bean
&&   \big\Vert \sum _{j=0}^{M_T}\sum _{t=j+1}^T  a_{k,j} \big( X_t X_{t-j} -\gamma (j)\big) \big\Vert _{m/2}\\
&=& \big\Vert \sum _{i,j=1,i-j\in\{0,\ldots ,M_T\}}^T a_{k,i-j}\big( X_i X_j-\gamma (i-j)\big)\big\Vert _{m/2}\\
&\le & \mbox{C}\cdot \Big( \sum_{h=0}^{\infty } \delta_m(h)\Big)^2 \max _{k=1,\ldots , N_T} \left(\sum_{j=0}^{M_T} a_{k,j}^2\right)^{1/2}\, \sqrt T.
\eean
For  \eqref{eq.var-lower},
	first note that Theorem~S.6 in the supplement of \citeasnoun{XiaoWu2014} assures summability of the joint fourth order cumulants.
	Hence, 
	\begin{equation}\label{eq.var-pr}
		\begin{aligned}
			&\big\Vert \frac{1}{\sqrt{T}}  \sum _{t=1}^T Z_{t,k} \big\Vert_2^2\\
			&=\frac{1}{4\pi^2TM_T}	\sum_{j_1,j_2=-M_T}^{M_T}\sum_{t_1= |j_1|+1}^T\sum_{t_2= |j_2|+2}^T
			w\big(\frac{j_1}{M_T}\big)w\big(\frac{j_2}{M_T}\big)\,\cos(j_1\lambda_{k,T})\,\cos(j_2\lambda_{k,T})\\
			&\qquad\times\gamma(t_1-t_2)\gamma(t_1-t_2+|j_2|-|j_1|)\\
			&\quad+\frac{1}{4\pi^2TM_T}	\sum_{j_1,j_2=-M_T}^{M_T}\sum_{t_1= |j_1|+1}^T\sum_{t_2= |j_2|+1}^T
			w\big(\frac{j_1}{T}\big)w\big(\frac{j_2}{T}\big)\,\cos(j_1\lambda_{k,T})\,\cos(j_2\lambda_{k,T})\\
			&\qquad\times \gamma(t_1-t_2+|j_2|)\gamma(t_1-t_2-|j_1|)+R^{(1)}_{T,k}
		\end{aligned}
	\end{equation}		
	with $\sup_k|R_{T,k}^{(1)}|=o(1)$. 
	For the first summand on the r.h.s.~we get from $\sum_{k \in \N} k|\gamma(k)|<\infty$ that
	$$
	\begin{aligned}
		&\frac{1}{4\pi^2TM_T}	\sum_{j_1,j_2=-M_T}^{M_T}\sum_{t_1 =|j_1|+1}^T\sum_{t_2 = |j_2|+1}^T
		w\big(\frac{j_1}{T}\big)w\big(\frac{j_2}{T}\big)\,\cos(j_1\lambda_{k,T})\,\cos(j_2\lambda_{k,T})\\
		&\qquad\times\gamma(t_1-t_2)\gamma(t_1-t_2+|j_2|-|j_1|) \\
		&=  \frac{1}{\pi^2TM_T}	\sum_{j_1,j_2=1}^{M_T}\sum_{t_1 =j_1+1}^T\sum_{t_2= j_2+1}^T
		w\big(\frac{j_1}{M_T}\big)w\big(\frac{j_2}{M_T}\big)\,\cos(j_1\lambda_{k,T})\,\cos(j_2\lambda_{k,T})\\
		&\qquad\times\gamma(t_1-t_2)\gamma(t_1-t_2+j_2-j_1)+R^{(2)}_{T,k}\\
		&=  \frac{1}{\pi^2 M_T}\sum_{t\in\Z}\sum_{ j\in \Z}\gamma(t)\gamma (t+j)	\sum_{j_1=1\vee (1-j)}^{M_T\wedge (M_T-j)} 
		w\big(\frac{j_1}{M_T}\big)w\big(\frac{j+j_1}{M_T}\big)\,\cos(j_1\lambda_{k,T})\\
		& \qquad \times \,\cos((j+j_1)\lambda_{k,T})+R^{(3)}_{T,k}\\
		&=  \frac{1}{2\pi^2 M_T}\sum_{t\in\Z}\sum_{ j\in \Z}\gamma(t)\gamma (t+j)	\sum_{j_1=1 }^{M_T } 
		w^2\big(\frac{j_1}{M_T}\big) \,\\
		& \qquad\times\, \big[\cos(j\lambda_{k,T})\,+\, \cos((j+2j_1)\lambda_{k,T})\big] +R^{(4)}_{T,k}\\
		&= f^2(\lambda _{k,T})   \, \frac{2}{  M_T} 	\sum_{j_1=1 }^{M_T } 
		w^2\big(\frac{j_1}{M_T}\big) \,(1+\cos(2j_1\lambda_{k,T}))  +R^{(5)}_{T,k},
	\end{aligned}
	$$
	where $\sup_{k}|R_{T,k}^{(\ell)}|=o(1),~\ell=2,\dots, 5.$ Similarly, 
	we obtain from $\sum_{k \in \N} k^2|\gamma(k)|<\infty$ for the second summand on the r.h.s.~of \eqref{eq.var-pr}
	$$
	\begin{aligned}
		&	\frac{1}{4\pi^2TM_T}	\sum_{j_1,j_2=-M_T}^{M_T}\sum_{t_1= |j_1|+1}^T\sum_{t_2> |j_2|+1}^T
		w\big(\frac{j_1}{T}\big)w\big(\frac{j_2}{T}\big)\,\cos(j_1\lambda_{k,T})\,\cos(j_2\lambda_{k,T})\\
		&\qquad\times \gamma(t_1-t_2+|j_2|)\gamma(t_1-t_2-|j_1|)\\
		&=\frac{1}{\pi^2M_T}\sum_{t\in\Z}\sum_{j=2}^{2M_T}\gamma(t)\gamma(t-j)\, \sum_{j_1=1\vee (j-M_T)}^{M_T\wedge (j-1)}	w\big(\frac{j_1}{T}\big)w\big(\frac{j-j_1}{T}\big)\,\cos(j_1\lambda_{k,T})\,\\
		&\qquad\times\,\cos((j-j_1)\lambda_{k,T})\, +R_{T,k}^{(6)}\\
		&	=\frac{1}{\pi^2M_T}\sum_{t\in\Z}\sum_{j=2}^{\sqrt{M_T}}\gamma(t)\gamma(t-j)\, \sum_{j_1=1}^{\sqrt{M_T}-1}	w\big(\frac{j_1}{T}\big)w\big(\frac{j-j_1}{T}\big)\,\cos(j_1\lambda_{k,T})\,\\
		&\qquad\times\,\cos((j-j_1)\lambda_{k,T})\, +R_{T,k}^{(7)}\\
		&= R_{T,k}^{(8)},
	\end{aligned}
	$$	
	where $\sup_{k}|R_{T,k}^{(\ell)}|=o(1),~\ell=6,\,7,\, 8.$ This finishes the proof of \eqref{eq.var-lower}.
	\\
Inequality \eqref{B-15a} in turn can be deduced from Proposition~1 in \citeasnoun{LiuWu2010} as follows	
$$
\Big\| \sum_{t=1}^T (Z_{t,k}-\widetilde Z_{t,k}^{(s)})\Big\|_{m/2}\leq \mbox{C}\,\sqrt T\, d_{s,m}\,  \max _{k=1,\ldots , N_T}\left(\sum_{j=0}^{M_T} a_{k,j}^2\right)^{1/2}\, \sum_{h=0}^{\infty } \delta_m(h).
$$

\end{proof}

\begin{proof}[Proof of Theorem~\ref{t.asymptotics}]
 We  mainly
follow the strategy of the proof of Theorem~1 in \citeasnoun{Zhang_etal2022} although some of the conditions used there are not fulfilled in our case.
In particular, a different $ m$-dependent approximation  $\widetilde{Z}^{(s)}_{t,k}$, i.e. $m=2s$,  is used in our proof which   then  leads  to  an improved rate of convergence.
Using the notation $h_{\tau,\tau,x}$ as in  \citeasnoun{Zhang_etal2022}, we  get the  bound
\begin{align} \label{B-14}
&\sup _{x\in \mathbb R} \Big\vert P\Big \{ \max _{k=1,\ldots , N_T} 
\sqrt{\frac{T}{M_T}} \Big\vert \widehat f_T(\lambda_{k,T} )   -E\, \widehat f_T(\lambda_{k,T} )\Big\vert \le x \Big\} \nonumber \\
&\quad \quad - P\Big\{ \max _{k=1,\ldots , N_T} \vert \xi_k \vert \le x \Big\} \Big\vert  \nonumber \\
 \leq &  \sup_{x\in\R}\,\Big|Eh_{\tau,\tau,x}\Big(\frac{1}{\sqrt T}\sum_{t=1}^T Z_{t,1},  \dots,\frac{1}{\sqrt T}\sum_{t=1}^T Z_{t,N_T} \Big)
 - Eh_{\tau,\tau,x}\left(\xi_{1},\dots, \xi_{N_T}\right)\Big| \nonumber \\
&\quad \quad+   C\, t\,  \left(1+\sqrt{\log(N_T)}  +   \sqrt{|\log(t)|}\right) ,
\end{align}
in view of  \eqref{eq.var_f} and \eqref{B-13}, where $t=(1+\log(2N_T))/\tau $ for some $\tau >0$  to be specified later.
For the $m$-dependent random variables (with $m=2s$) $ \widetilde Z_{t,k}^{(s)}$, $ k=1,2, \ldots, N_T$,
defined in \eqref{eq.Z-tilde},
by  the properties of the function $h_{\tau,\tau, x}$  and  from  Lemma~\ref{l.inequalities}, we get
\begin{align} \label{B-15b}
&\sup_{x\in\R}\,\Big|Eh_{\tau,\tau,x}\Big(\frac{1}{\sqrt T}\sum_{t=1}^T Z_{t,1},\dots,\frac{1}{\sqrt T}\sum_{t=1}^T Z_{t,N_T} \Big)
\nonumber \\
& \ \ \ \ \ \ \ \ \ \ \ \ \ \ \ \ - Eh_{\tau,\tau,x}\Big(\frac{1}{\sqrt T}\sum_{t=1}^T \widetilde Z^{(s)}_{t,1},\dots,\frac{1}{\sqrt T}\sum_{t=1}^T Z^{(s)}_{t,N_T} \Big)\Big| \nonumber \\
&\leq C\, \tau\, E\Big[\max_{k=1,\dots, N_T}\Big |\frac{1}{\sqrt T}\sum_{t=1}^T (Z_{t,k}-\widetilde Z_{t,k}^{(s)})\Big |\Big] \nonumber \\
&\leq C\, \tau\, N_T^{2/m}\, \max_{k=1,\dots, N_T}
\Big \|\frac{1}{\sqrt T}\sum_{t=1}^T (Z_{t,k}-\widetilde Z_{t,k}^{(s)})\Big\|_{m/2 \nonumber }\\
&\leq C\, \tau\, N_T^{2/m} \, d_{s,m}
\end{align}
with $d_{s,m}$ given in  (\ref{eq.dsm}). The introduction of  the random variables $ \widetilde{Z}^{(s)}_{t,k}$ allows  us to proceed with the classical big block/small block  technique. Toward this, we define for any integer $l>2s$ big and small blocks  of the $ \widetilde{Z}_{t,k}^{(s)}$ as 
$$
S_{j,k}\,:=\, \frac{1}{\sqrt T}\sum_{t=2(j-1)(s+l)+1}^{2(j-1)(s+l)+2l}\widetilde Z_{t,k}^{(s)}
\quad\text{and}\quad 
U_{j,k}\,:=\, \frac{1}{\sqrt T}\sum_{t=2(j-1)(s+l)+2l+1}^{(2j(s+l))\wedge T}\widetilde Z_{t,k}^{(s)}
$$
for $j=1,\dots, \lceil \frac{T}{2(s+l)}\rceil=: V_T$.
Smoothness of $h_{x,x,\tau}$,  Rosenthal's inequality and similar arguments as in  \eqref{B-13} yield 
\begin{align} \label{B-17}
	&\sup_{x\in\R}\,\Big|Eh_{\tau,\tau,x}\Big(\frac{1}{\sqrt T}\sum_{t=1}^T   \widetilde Z^{(s)}_{t,1},\dots,\frac{1}{\sqrt T}\sum_{t=1}^T Z^{(s)}_{t,N_T} \Big) \nonumber \\
	 & \ \ \ \ \ \ \ \ \ \ \ \ \ \ \ \ - Eh_{\tau,\tau,x}\Big(\sum_{j=1}^{V_T} S_{j,1},\dots,\sum_{j=1}^{V_T} S_{j,N_T}\Big)\Big| \nonumber \\
	&\leq C\, \tau\, N_T^{2/m}\, \max_{k=1,\dots, N_T} \Big\|\sum_{j=1}^{V_T}U_{j,k}\Big\|_{m/2} \nonumber \\
	&\leq C\, \tau\, N_T^{2/m}\,\sqrt{V_T}\,\max_{k=1,\dots , N_T}\max_{j=1,\dots , V_T}\|U_{j,k}\|_{m/2} \nonumber \\
	&\leq C\,\tau\, N_T^{2/m}\,\sqrt{\frac{V_T\, s}{ T }} \nonumber \\
	&\leq  C\,\tau\, N_T^{2/m}\,\sqrt{\frac{  s}{l }}.
\end{align} 
Next, we define centred, joint normal random variables $(S^*_{j,k})_{j=1,\dots, V_T, k=1,\dots, N_T}$ with $E[S^*_{j,k_1}S^*_{j,k_2}]=E[S_{j,k_1}S_{j,k_2}]$ and such that $(S^*_{j,1},\dots, S^*_{j,N_T}),$ $j=1,\dots, V_T$ are independent. Further these variables are constructed such that they are independent of  $(S_{j,k})_{j=1,\dots, V_T, k=1,\dots, N_T}$.
We obtain from Lindeberg's method
and  using
 \[ ||S_{j,k}^\ast\|_{m/2} \leq  C_m\|S^\ast_{j,k}\|_2 =  C_m \|S_{j,k}\|_2\leq C_m \|S_{j,k}\|_{m/2},\]
 that, 
\begin{equation}\label{B-18_and_B-19}
	\begin{aligned}
		&\sup_{x\in\R}\,\Big|Eh_{\tau,\tau,x}\Big(\sum_{j=1}^{V_T} S_{j,1},\dots,\sum_{j=1}^{V_T} S_{j,N_T} \Big)
		- Eh_{\tau,\tau,x}\Big(\sum_{j=1}^{V_T} S_{j,1}^*,\dots,\sum_{j=1}^{V_T} S_{j,N_T}^*\Big)\Big|\\
		&\leq C\, (1+\tau^3)\,  \sum_{j=1}^{V_T} E\max_{k=1,\dots, N_T} [|S_{j,k}|^3 \,+\, |S_{j,k}^*|^3]\\
		&\leq C^{\prime }_m\, (1+\tau^3)\,\, N_T^{6/m}\, V_T\,\left(\frac{l}{T}\right)^{3/2} . 
	\end{aligned}
\end{equation} 
Using the bounds
\[ \|\frac{1}{\sqrt{T}} \sum_{i=1}^T (Z_{i,k} - \widetilde{Z}_{i,k}^{(s)})\|_{m/2} \leq C d_{s,m},  \   \  \|\frac{1}{\sqrt{T}}\sum_{i=1}^T Z_{i,k}\|_{m/2} \leq C,\]
derived in Lemma~\ref{l.inequalities},
\begin{align*}
|\sum_{j_1=1}^{V_T} \sum_{j_2=1}^{V_T} E S_{j_1,k_1} U_{j_2,k_2}| & \leq \|\sum_{j=1}^{V_T}  S_{j,k_1}\|_{m/2} \| \sum_{j=1}^{V_T}U_{j,k_2}\|_{m/2} \\
& \leq V_T \max_{j=1,\ldots, V_t}\|S_{j,k_1}\|_{m/2}\max_{j=1,\ldots, V_t}\|U_{j,k_2}\|_{m/2} \\
& \leq C\,V_{T}\, \frac{\sqrt{l s}}{T}
\end{align*}
and 
\[ |\sum_{j_1=1}^{V_T} \sum_{j_2=1}^{V_T} E U_{j_1,k_1} U_{j_2,k_2}| \leq V_T \max_{j=1,\ldots, V_T} \|U_{j,k_1}\|_{m/2}   \max_{j=1,\ldots, V_T} \|U_{j,k_2}\|_{m/2} \leq C V_T \frac{s}{T}, \]
we get 
\begin{equation} \label{B-23}
\begin{aligned}
 \sup_{x\in\R}\,\Big|Eh_{\tau,\tau,x}\Big(\sum_{j=1}^{V_T} S^\ast_{j,1},\dots,\sum_{j=1}^{V_T} S^\ast_{j,N_T} \Big)
		& - Eh_{\tau,\tau,x}\Big(\xi_1,\ldots, \xi_{N_T}\Big)\Big| \\
		& \leq C\,\tau^2\,\big(d_{s,m} + V_T \frac{\sqrt{s l}}{T} \big) .
\end{aligned}
\end{equation}
Now, let $ \tau \sim T^{\lambda}$, $ s\sim T^{a_s}$ and $ l \sim T^{a_l}$  for  some positive numbers  $\lambda $, $a_s$ and $  a_l $. For  (\ref{B-14}), (\ref{B-15b}), (\ref{B-17}), (\ref{B-18_and_B-19}) and (\ref{B-23})  to vanish asymptotically, 
 the following conditions are sufficient:
\begin{enumerate}
\item[(i)] $ t \big(1 +\sqrt{\log(N_T)} + \sqrt{|\log(t)|}\big) \rightarrow 0$, 
\item[(ii)] $ 2\lambda + 4/m +a_s <a_l $, \  $ 6\lambda + 12/m +a_l <1  $,  \  $ 4\lambda + a_s <a_l$ and
\item[(iii)] $ \big(T^{\lambda +2/m} + T^{2\lambda}\big)d_{s,m} \rightarrow 0$.
\end{enumerate}
For (i) and because   $ t=(1+\log(2N_T))/\tau$,  it is easily seen  that   
\[  t \big(1 +\sqrt{\log(N_T)} + \sqrt{|\log(t)|}\big) \leq C T^{-\lambda} \big(\log(N_T)\big)^{3/2},\]
 which converges to zero  provided  $\lambda >0$. 
Furthermore, since  $ m>16$, (ii) is satisfied  if  
$a_s +\max\{4\lambda, 2\lambda+4/m\} < a_l < 1-6\lambda -12/m$. 
Finally,  for (iii) notice first that for $ 0<\delta <\alpha-1$ and  because $ s\sim T^{a_s}$,
\begin{align*}
\big(\sum_{j=s+1}^\infty \delta^2_m(j)\big)^{1/2} & = \big(\sum_{j=s+1}^\infty j^{-2\alpha +1+2\delta} \cdot j^{-(1+2\delta)}\big)^{1/2}\\
& \leq C s^{-\alpha +\delta +1/2} \leq C T^{-a_s(\alpha-\delta -1/2)}.
\end{align*}
Now, for  $r>0$  we have,
\begin{align*}
&\sum_{h=0}^\infty \min\{ (1+h)^{-\alpha}, T^{-a_s( \alpha-1/2-\delta)}\}\\ & \leq  \sum_{h=0}^{T^r} T^{- a_s(\alpha-1/2 -\delta)} + \sum_{h=T^r}^\infty  (1+h)^{-\alpha +1+\delta} h^{-1-\delta} \\
& \leq T^{r -a_s(\alpha -1/2 -\delta)} + C T^{-r(\alpha-1-\delta)}.
\end{align*} 
Balancing both terms in the  last bound  above yields,
\[ r-a_s(\alpha -1/2-\delta) = -r(\alpha-1-\delta) \ \Longleftrightarrow  r =\frac{a_s(\alpha-1/2-\delta)}{\alpha-\delta}\]
and, therefore,
\[  \sum_{h=0}^\infty \min\big\{\delta_m(h), \big(\sum_{j=s+1}^\infty \delta^2_m(j)\big)^{1/2}\big\} =   T^{-\frac{\displaystyle a_s(\alpha-1/2-\delta)(\alpha-1-\delta)}{\displaystyle \alpha-\delta} }\rightarrow 0.\]
Hence for  (iii) to be satisfied, 
\begin{equation} \label{eq.for3}
\frac{\displaystyle a_s(\alpha-1/2-\delta)(\alpha-1-\delta)}{\displaystyle \alpha-\delta } -  \max\{\lambda +2/m \,,\, 2\lambda\} >0
\end{equation} 
should hold true.  
Notice that  for $\alpha  $ satisfying condition   (\ref{eq.a_s}),   (\ref{eq.for3}) holds true and at the same time the rate in (\ref{eq.for3}) is larger or equal  to $ \min\{\kappa_1, \kappa_2\}$. 
\end{proof}
\medskip

\begin{proof}[Proof of Theorem~\ref{t.conf-bands}]
Let 
$$R_T(\lambda_{k,T})=	\sqrt{\frac{T}{M_T}} \,\frac{\widehat f_T(\lambda_{k,T})- E\widehat f_T(\lambda_{k,T})}{f(\lambda_{k,T})}\,\left(\frac{f(\lambda_{k,T})}{\widehat f_T(\lambda_{k,T}) }-1\right)
$$
and  $ t=(1+\log(2N_T))/\tau$ with $\tau\sim T^\lambda$. Then we can split up
$$	
\begin{aligned}
	&	 P\big(   \max _{k=1,\ldots , N_T} \vert \widetilde\xi_k \vert \le x \big) 
	- P\big( \max _{k=1,\ldots , N_T}  
	\sqrt{\frac{T}{M_T}} \frac{\vert \widehat f_T(\lambda_{k,T}) - E\widehat f_T(\lambda_{k,T})\vert }{\widehat f_T(\lambda_{k,T})} \le x \big) \\
	&\leq P\big(   \max _{k=1,\ldots , N_T} \vert \widetilde\xi_k \vert \le x \big)- P\big(   \max _{k=1,\ldots , N_T} \vert \widetilde\xi_k \vert \le x-t \big)\\
&\quad +\,  P\big(   \max _{k=1,\ldots , N_T} \vert \widetilde\xi_k \vert \le x -t\big)	-   P\big( \max _{k=1,\ldots , N_T}  
	\sqrt{\frac{T}{M_T}} \frac{\vert \widehat f_T(\lambda_{k,T}) - E\widehat  f_T(\lambda_{k,T})\vert }{ f(\lambda_{k,T})} \le x -t \big) \\
	&\quad
	+\, P\big( \max _{k=1,\ldots , N_T} |R_T(\lambda_{k,T})|>t\big).
\end{aligned}	
$$
One can proceed analogously to obtain a lower bound which then results in 
$$	
\begin{aligned}
	&	\sup_{x\in\R} \big|\,P\big(   \max _{k=1,\ldots , N_T} \vert \widetilde\xi_k \vert \le x \big) 
	- P\big( \max _{k=1,\ldots , N_T}  
	\sqrt{\frac{T}{M_T}} \frac{\vert \widehat f_T(\lambda_{k,T}) -  E\widehat f_T(\lambda_{k,T})\vert }{\widehat f_T(\lambda_{k,T})} \le x \big)\,\big| \\
	&\leq 	\sup_{x\in\R} \big| P\big(   \max _{k=1,\ldots , N_T} \vert \widetilde\xi_k \vert \le x \big)- P\big(   \max _{k=1,\ldots , N_T} \vert \widetilde\xi_k \vert \le x-t \big)\big|\\
	&\quad +\, \sup_{x\in\R} \big|  P\big(   \max _{k=1,\ldots , N_T} \vert \widetilde\xi_k \vert \le x \big)	-   P\big( \max _{k=1,\ldots , N_T}  
	\sqrt{\frac{T}{M_T}} \frac{\vert \widehat f_T(\lambda_{k,T}) -E\widehat  f_T(\lambda_{k,T})\vert }{ f(\lambda_{k,T})} \le x  \big) |\\
	&\quad
	+\,\sup_{x\in\R} \big|  P\big( \max _{k=1,\ldots , N_T} |R_T(\lambda_{k,T})|>t\big)\big|\\
	&=P_1+P_2+P_3
\end{aligned}	
$$
with obvious abbreviations for $P_1,~P_2, $ and $P_3$. Lemma~A.1 in \citeasnoun{Zhang_etal2022} gives
$$
P_1 \,\leq\, C\, t\,  \left(1+\sqrt{\log(N_T)}  +   \sqrt{|\log(t)|}\right).
$$
 The desired rate for $P_2$ can be derived with exactly the same arguments as in the proof of Theorem~\ref{t.asymptotics} 
 since the spectral density is assumed to be uniformly bounded from below.

\noindent
 Finally, for $P_3$ we make use of the assumed bias property \eqref{eq.thm2a} of $\widehat f_T$, that is of
$$
\sup_\lambda|E\widehat f_T(\lambda)-f(\lambda)|=O(M_T^{-2}).
$$
This allows to bound $P_3$ as follows.
We have
\begin{align*}
\vert R_T(\lambda_{k,T})\vert &=	\sqrt{\frac{T}{M_T}} \frac{\vert \widehat f_T(\lambda_{k,T})- E\widehat f_T(\lambda_{k,T})\vert}{f(\lambda_{k,T})}\,\Big\vert \frac{f(\lambda_{k,T})}{\widehat f_T(\lambda_{k,T}) }-1\Big\vert\\
&=\sqrt{\frac{T}{M_T}} \,\big\vert \widehat f_T(\lambda_{k,T})- E\widehat f_T(\lambda_{k,T})\big\vert
\frac{\vert \widehat f_T(\lambda _{k,T}) -f(\lambda _{k,T})\vert }{\widehat f_T(\lambda _{k,T}) \, f(\lambda _{k,T})}\, .
\end{align*}
A division of the considerations depending on whether 
$ \max _{k} \vert \widehat f_T(\lambda_{k,T}) -f(\lambda _{k,T})\vert $ is less or equal or larger 
than $\inf _{\lambda} f(\lambda )/2$ results in
$$
\begin{aligned}
&P(\max_{k=1,\dots, N_T}|R_T(\lambda_{k,T})|>t)\\
&\leq P\Big( \sqrt{\frac{T}{M_T}}\max_{k=1,\dots, N_T} \vert \widehat f_T(\lambda_{k,T})- E\widehat f_T(\lambda_{k,T})\vert
\frac{2}{\inf _{\lambda} f(\lambda )^2}
\vert \widehat f_T(\lambda_{k,T})- f(\lambda_{k,T})\vert >t \Big)\\
&\quad + P\big( \max _{k=1,\dots, N_T} \vert \widehat f_T(\lambda_{k,T}) -f(\lambda _{k,T})\vert >
\inf _{\lambda} f(\lambda )/2 \big)\, .
\end{aligned}
$$
The first summand is bounded through
\begin{align*}
&C t^{-1}\sqrt{\frac{T}{M_T}}\Big( E \max_{k=1,\dots, N_T}\vert \widehat f_T(\lambda_{k,T})- 
E\widehat f_T(\lambda_{k,T})\vert^2 \\
&\hspace*{1.58cm} + E \max_{k=1,\dots, N_T}\vert \widehat f_T(\lambda_{k,T})- 
E\widehat f_T(\lambda_{k,T})\vert \cdot M_T^{-2}\Big)\\
&\leq \, C t^{-1}\Big( N_T^{4/m} \sqrt{\frac{M_T}{T}} + N_T^{2/m}M_T^{-2}\Big)\, ,
\end{align*}
where for arbitrary random variables $U_k$ we make use of 
$E \max_{k=1,\dots, N_T}\vert U_k\vert \le (E\sum _{k=1}^{N_T} \vert U_k\vert ^r)^{1/r} \le
N_T^{1/r} \max_{k=1,\dots, N_T}\Vert U_k\Vert _r$, $r\ge 1,$ and   
 the last inequality follows from Lemma~\ref{l.inequalities}.\\
A similar consideration for the second summand finally leads to
\[P(\max_{k=1,\dots, N_T}|R_T(\lambda_{k,T})|>t)
\leq\,  C (t^{-1}\, N_T^{4/m}+ N_T^{2/m})\,\Big(\sqrt{\frac{M_T}{T}} + M_T^{-2}\Big) \, .
\]

Using  $ M_T \sim T^{a_s}$  and $ t^{-1} =T^{\lambda}\big/(1+\log(2N_T))$, the above bound for $ P(\max_{k=1,\dots, N_T}|R_T(\lambda_{k,T})|>t)$ implies the order 
\[ 
\frac{\displaystyle T^{\lambda+a_s/2} T^{4/m}}{\displaystyle T^{1/2}(1+\log(2N_T)) } + \frac{\displaystyle T^\lambda T^{4/m}}{\displaystyle T^{2a_s} (1+\log(2N_T))},\]
for  $P_3$, which  for $\lambda+\kappa\leq \min\{1/2-a_s/2-4/m\,,\, 2a_s-4/m\}$  leads to $P_3=O(T^{-\kappa}/\log(2N_T))$.
\end{proof}

\begin{proof}[Proof of Proposition~\ref{l.sigma-hat}]
	First, we split up
	\begin{align}  \label{eq.Boundsigma}
		&\big |\widehat{\sigma}_T(j_1,j_2)\, -\, \sigma_T(j_1,j_2)\big| \nonumber \\
		\leq \ & \Big|  \frac{1}{T}\sum _{t=j_1+1}^T  \sum _{s=j_2+1}^T K\Big(\frac{t-s}{b_T}\Big)\big( X_t X_{t-j_1} -{\gamma} (j_1)\big) 
		\big( X_s X_{s-j_2} - {\gamma} (j_2)\big) \,-\, \sigma(j_1,j_2)\Big| \nonumber \\
		& +\, |\widehat\gamma(j_1)-\gamma(j_1)|\,\Big| \frac{1}{T}\sum _{t=j_1+1}^T  \sum _{s=j_2+1}^T K\Big(\frac{t-s}{b_T}\Big)\big( X_s X_{s-j_2} - {\gamma} (j_2)\big)\Big| \nonumber \\
		& +\, |\widehat\gamma(j_2)-\gamma(j_2)|\,\Big| \frac{1}{T}\sum _{t=j_1+1}^T  \sum _{s=j_2+1}^T K\Big(\frac{t-s}{b_T}\Big)\big( X_t X_{t-j_1} - {\gamma} (j_1)\big)\Big| \nonumber \\
		&+|\widehat\gamma(j_1)-\gamma(j_1)|\,|\widehat\gamma(j_2)-\gamma(j_2)|\, \Big| \frac{1}{T}\sum _{t=j_1+1}^T  \sum _{s=j_2+1}^T K\Big(\frac{t-s}{b_T}\Big)\Big|. 
	\end{align}
	While the last summand is of order $O(b_T\, T^{-1})$ (see (E.4) in \citeasnoun{Zhang_etal2022}), the two middle terms can be bounded from above by 
	$
	O(b_T\, T^{-1/2})
	$
	using Cauchy-Schwarz inequality. Both bounds hold uniformly in $j_1$ and $j_2$.
Let $Y_{t,j}=X_tX_{t-j} - \gamma(j)$ and
	\[
	 \widetilde{\sigma}_T(j_1,j_2) = \frac{1}{T}\sum_{t=j_1+1}^T\sum_{s=j_2+1}^T K\Big(\frac{t-s}{b_T}\Big) Y_{t,j_1} Y_{s,j_2}.
	 \]
	Then the first term on the right hand side of the bound given in (\ref{eq.Boundsigma}) equals $ |\widetilde{\sigma}_T(j_1,j_2) -\sigma(j_1,j_2) |$,  and for this we have  
 \begin{align} \label{eq.sigma12}
 |\widetilde{\sigma}(j_1,j_2) -\sigma(j_1,j_2) | & \leq \Big|\frac{1}{T}\sum_{t=j_1+1}^T\sum_{s=j_2+1}^T \Big(K\Big(\frac{t-s}{b_T}\Big)-1\Big) E(Y_{t,j_1}Y_{s,j_2}) \Big|\nonumber \\
 & + \Big|\frac{1}{T}\sum_{t=j_1+1}^T\sum_{s=j_2+1}^T \Big( Y_{t,j_1}Y_{s,j_2}-E(Y_{t,j_1}Y_{s,j_2}) \Big)K\Big(\frac{t-s}{b_T}\Big)  \Big|.
 \end{align}
Using Lemma~\ref{l.cov}(iii)  
 the first term of the right hand side of  (\ref{eq.sigma12}) can be bounded by 
\begin{align*}
C\frac{1}{T}\sum_{t=1}^T\sum_{s=1}^T & \Big(1-K\Big(\frac{t-s}{b_T}\Big) \Big)\frac{1}{(1+|t-s|)^{\alpha}}  \leq 2C\sum_{s=0}^{\infty} 
\Big(1-K\Big(\frac{s}{b_T}\Big) \Big) \frac{1}{(1+s)^{\alpha}}.
\end{align*}
	Let $ S= b_T $.  Using $ |1-K(s/b_T)| \leq  \sup_{u\in [0,1] } |K^{\prime}(u)| s/b_T$ for $ 0\leq s\leq S$ and $ |1-K(s/b_T)| \leq 1$ for $ s\geq S+1$, we get 
\begin{align*}
\sum_{s=0}^{\infty}  \Big(1- & K\Big(\frac{s}{b_T}\Big) \Big) \frac{1}{(1+s)^{\alpha}} \\
& \leq  \frac{\sup_{u\in[0,1]}|K^{\prime}(u)|}{b_T} \sum_{s=0}^S \frac{1}{(1+s)^{\alpha-1}} +
\sum_{s=S+1}^\infty  \frac{1}{(1+s)^{\alpha}},
\end{align*}
where the first term is $ O(b_T^{-1})$ because   $ \alpha >2$. For the second term we get  
\begin{align*}
\sum_{s=S+1}^\infty  \frac{1}{(1+s)^{\alpha}} & \leq \int_{S+1}^\infty \frac{1}{x^{\alpha}}dx =\frac{1}{(\alpha-1)(S+1)^{\alpha-1}} =O(b_T^{-1}).
\end{align*} 
Hence  
\[ \max_{1\leq j_1,j_2 \leq M_T}\Big|\frac{1}{T}\sum_{t=j_1+1}^T\sum_{s=j_2+1}^T \Big(K\Big(\frac{t-s}{b_T}\Big)-1\Big) E(Y_{t,j_1}Y_{s,j_2}) \Big| =O(b_T^{-1}).\]

Consider  next the second term  of the bound given in  (\ref{eq.sigma12}) and observe that this term can be   bounded by
\begin{equation} \label{eq.Term2.1}
 M_T^{8/m}\max_{1\leq j_1,j_2\leq M_T}  \Big\|\frac{1}{T}\sum_{t=j_1+1}^T\sum_{s=j_2+1}^T \Big( Y_{t,j_1}Y_{s,j_2}-E(Y_{t,j_1}Y_{s,j_2}) \Big)K\Big(\frac{t-s}{b_T}\Big)  \Big\|_{m/4} ,
\end{equation}
where
\begin{align*}
\Big\|\frac{1}{T}\sum_{t=j_1+1}^T\sum_{s=j_2+1}^T \Big(  & Y_{t,j_1}Y_{s,j_2}-  E(Y_{t,j_1}Y_{s,j_2}) \Big)K\Big(\frac{t-s}{b_T}\Big)  \Big\|_{m/4} \\
\leq &  \frac{1}{T}\sum_{\ell=0}^{T-1} K(\ell/b_T)\big\|\sum_{t=1}^{T-\ell} ( Y_{t,j_2}Y_{t+\ell,j_1}-E(Y_{t,j_2}Y_{t+\ell,j_1}) ) \big\|_{m/4} \\
& \ \  +  \frac{1}{T}\sum_{\ell=1}^{T-1} K(\ell/b_T)\big\|\sum_{s=1}^{T-\ell} ( Y_{s,j_1}Y_{s+\ell,j_2}-E(Y_{s,j_1}Y_{s+\ell,j_2}) ) \big\|_{m/4}.
\end{align*}
Recall that for  $s\geq  0$, $ {\mathcal F}_{r,s} $ denotes the $\sigma$-algebra generated by the set of random variables $ \{e_r,e_{r-1}, \ldots, e_{r-s} \}$. 
Note, that by Assumption A.1 for some measurable function $H$,
$ Y_{t,j_2}Y_{t+\ell,j_1} = H(e_{t+\ell}, e_{t+\ell-1}, \ldots)$  and that the $ e_t$'s are i.i.d.. Then  
\begin{align*}
 Y_{t,j_2}Y_{t+\ell,j_1}  & -E(Y_{t,j_2}Y_{t+\ell,j_1})  = E( Y_{t,j_2}Y_{t+\ell,j_1} | {\mathcal F}_{t+\ell, \infty} )- E(Y_{t,j_2}Y_{t+\ell,j_1}) \\
& = \lim_{q\rightarrow\infty} E( Y_{t,j_2}Y_{t+\ell,j_1} | {\mathcal F}_{t+\ell, q}) - E(Y_{t,j_2}Y_{t+\ell,j_1}) \\
& = E( Y_{t,j_2}Y_{t+\ell,j_1} | {\mathcal F}_{t+\ell, 0} )- E(Y_{t,j_2}Y_{t+\ell,j_1})  \\
 & \ \ \  \  + \sum_{r=1}^\infty \big\{ E( Y_{t,j_2}Y_{t+\ell,j_1} | {\mathcal F}_{t+\ell, r}) - E( Y_{t,j_2}Y_{t+\ell,j_1} | {\mathcal F}_{t+\ell, r-1} )\big\},
\end{align*}
a.s. 
Therefore,
\begin{align*}
\big\|\sum_{t=1}^{T-\ell} ( & Y_{t,j_2}Y_{t+\ell,j_1}-E(Y_{t,j_2}Y_{t+\ell,j_1}) ) \big\|_{m/4}\\
\leq & \big\|\sum_{t=1}^{T-\ell} \big\{E( Y_{t,j_2}Y_{t+\ell,j_1} | {\mathcal F}_{t+\ell, 0} )- E(Y_{t,j_2}Y_{t+\ell,j_1}) \big\}  \big\|_{m/4}\\
& +  \sum_{r=1}^\infty \big\|\sum_{t=1}^{T-\ell}\big\{ E( Y_{t,j_2}Y_{t+\ell,j_1} | {\mathcal F}_{t+\ell, r}) - E( Y_{t,j_2}Y_{t+\ell,j_1} | {\mathcal F}_{t+\ell, r-1} )\big\}\big\|_{m/4}\\
 = & S_{1,n} + S_{2,n},
\end{align*}
with an obvious notation for $ S_{1,n}$ and $S_{2,n}$.
Consider $S_{1,n}$ and observe that  $ E( Y_{t,j_2}Y_{t+\ell,j_1} | {\mathcal F}_{t+\ell, 0} ) =\widetilde{g}(e_{t+\ell})$ and  therefore,  
$E( Y_{t,j_2}Y_{t+\ell,j_1} | {\mathcal F}_{t+\ell, 0} )$ and \\ $  E( Y_{s,j_2}Y_{s+\ell,j_1} | {\mathcal F}_{s+\ell, 0} ) $ are independent for $ t\neq s$. Hence,
\begin{align*}
S_{1,n} &\leq C \sqrt{\sum_{t=1}^{T-\ell} \big\|\big\{E( Y_{t,j_2}Y_{t+\ell,j_1} | {\mathcal F}_{t+\ell, 0} )- E(Y_{t,j_2}Y_{t+\ell,j_1}) \big\}  \big\|_{m/4}}\\
& \leq C\sqrt{T} \big\| E( Y_{t,j_2}Y_{t+\ell,j_1} | {\mathcal F}_{t+\ell, 0} )- E(Y_{t,j_2}Y_{t+\ell,j_1}) \big\|_{m/4}.
\end{align*}
To bound the term  $S_{2,n}$ define first 
\[  W_u =\sum_{t=T-\ell-u+1}^{T-\ell} \big\{ E( Y_{t,j_2}Y_{t+\ell,j_1} | {\mathcal F}_{t+\ell, r}) - E( Y_{t,j_2}Y_{t+\ell,j_1} | {\mathcal F}_{t+\ell, r-1} ) \big\} \]
and denote by ${\mathcal A}_u$ the $\sigma$-algebra generated by the set  $ \{e_{T}, e_{T-1}, \ldots, e_{T-u+1-r}\}$. Notice that $ W_u$ is measurable with respect to ${\mathcal A}_u$ and that $ {\mathcal A}_u\subset {\mathcal A}_{u+1}$. 
Furthermore,
\begin{align*}
E(& W_{u+1}  -W_u | {\mathcal A}_u) \\
& = E\Big[ E\big(Y_{T-\ell-u,j_2}Y_{T-u,j_1}\vert {\mathcal F}_{T-u,r}\big) -  E\big(Y_{T-\ell-u,j_2}Y_{T-u,j_1}\vert {\mathcal F}_{T-u,r-1}\big)\Big\vert {\mathcal A}_u\Big]\\
& = E\big(Y_{T-\ell-u,j_2}Y_{T-u,j_1}\vert {\mathcal F}_{T-u,r-1}\big) -  E\big(Y_{T-\ell-u,j_2}Y_{T-u,j_1}\vert {\mathcal F}_{T-u,r-1}\big)\\
& = 0.
\end{align*}
Since $ W_u $ forms a martingale we get 
\[ 
\|W_{T-\ell}\|_{m/4} \leq C\sqrt{ \sum_{t=1}^{T-\ell} \|E(Y_{t,j_2}Y_{t+\ell,j_1}|{\mathcal F}_{t+\ell,r}) -E(Y_{t,j_2}Y_{t+\ell,j_1}|{\mathcal F}_{t+\ell,r-1}) \|^2_{m/4}}.
\]
Recall that $ Y_{t,j_2}Y_{t+\ell,j_1} = H(e_{t+\ell}, e_{t+\ell-1}, \ldots)$ for some measurable function $H$ and let
\[ 
Y_{t,j_2}(r-\ell)Y_{t+\ell,j_1}(r) = H(e_{t+\ell}, e_{t+\ell-1}, \ldots, e_{t+\ell-r+1}   , e^{\prime}_{t+\ell-r}, e_{t+l-r-1}, \ldots  ),
\]
where $ \{e_t^{\prime},t\in \Z \}$ is an independent copy of $ \{e_t,t\in \Z\}$.
Then,
\begin{align*} 
\|E(Y_{t,j_2}Y_{t+\ell,j_1}| {\mathcal F}_{t+\ell,r})&  - E(Y_{t,j_2}Y_{t+\ell,j_1}| {\mathcal F}_{t+\ell,r-1})\|_{m/4}\\
& = \|E(Y_{t,j_2}Y_{t+\ell,j_1}| {\mathcal F}_{t+\ell,r})  - E(Y_{t,j_2}(r-\ell)Y_{t+\ell,j_1}(r)| {\mathcal F}_{t+\ell,r})\|_{m/4}\\
&\leq \| Y_{t,j_2}Y_{t+\ell,j_1}-  Y_{t,j_2}(r-\ell)Y_{t+\ell,j_1}(r)\|_{m/4},
\end{align*} 
that is, 
\[  
\|W_{T-\ell}\|_{m/4} \leq C\sqrt{T}  \|Y_{t,j_2}Y_{t+\ell,j_1}-  Y_{t,j_2}(r-\ell)Y_{t+\ell,j_1}(r)\|_{m/4}.
\]
Hence,
\begin{align*}
S_{2,n}=&\sum _{r=1}^{\infty} \| W_{T-\ell} \| _{m/4} \\
\leq &\ C\sqrt{T} \sum_{r=1}^{\infty}   \|Y_{t,j_2}Y_{t+\ell,j_1}  -  Y_{t,j_2}(r-\ell)Y_{t+\ell,j_1}(r)\|_{m/4} \\
\leq & C\sqrt{T} \Big\{  \sum_{r=1}^{\infty}   \|Y_{t+\ell,j_1}\|_{m/2} \|  Y_{t,j_2}(r-\ell) - Y_{t,j_2}\|_{m/2}\\
& +  \sum_{r=1}^{\infty}   \|Y_{t,j_2} (r-\ell)\|_{m/2}  \| Y_{t+\ell,j_1}(r) - Y_{t+\ell,j_1}\|_{m/2} \Big\}
 \ \leq \ C\sqrt{T} , 
\end{align*}
since 
\begin{align*}
 \|  Y_{t,j_2}(r-\ell) - Y_{t,j_2}\|_{m/2} & = \| X_t(r-\ell)X_{t-j_2}(r-\ell-j_2) - X_{t}X_{t-j_2}\|_{m/2}\\
  &  \leq  \| X_{t}(r-\ell)-X_t\|_m\|X_{t-j_2}(r-\ell-j_2) \|_m \\
  & \ \  + \| X_t \|_m \| X_{t-j_2}(r-\ell-j_2) -X_{t-j_2} \|_m\\
  & \leq C(r-\ell)^{\alpha}
\end{align*}
and an analogue argument for $ \| Y_{t+\ell,j_1}(r) - Y_{t+\ell,j_1}\|_{m/2}$.  Therefore,
\begin{align*}
&\Big\|\frac{1}{T}\sum_{t=j_1+1}^T\sum_{s=j_2+1}^T 
\Big(   Y_{t,j_1}Y_{s,j_2}-  E(Y_{t,j_1}Y_{s,j_2}) \Big)K\Big(\frac{t-s}{b_T}\Big)  \Big\|_{m/4} \\
\leq & \ \frac{1}{T}\sum_{\ell=0}^{T-1}K(\ell/b_T) \|\sum_{t=1}^{T-\ell} Y_{t,j_2}Y_{t+\ell,j_1} - E(Y_{t,j_2}Y_{t+\ell,j_1})\|_{m/4} \\
& + \frac{1}{T}\sum_{\ell=0}^{T-1}K(\ell/b_T) \|\sum_{t=1}^{T-\ell} Y_{t,j_1}Y_{t+\ell,j_2} - E(Y_{t,j_2}Y_{t+\ell,j_2})\|_{m/4} \\
\leq & \ CT^{-1/2} \sum_{\ell=0}^{T-1}K(\ell/b_T)\\
\leq & \ C(T^{-1/2}\big( K(0)+\sum_{\ell=1}^{\infty} K(\ell/b_T)\big)\\
\leq & \ CT^{-/2}\big( K(0) + \int_0^\infty K(x/b_T)dx\big)\\
\leq  & \ Cb_T/\sqrt{T}.
\end{align*}
From this and  recalling (\ref{eq.Term2.1}) we conclude that 
\[ \Big|\frac{1}{T}\sum_{t=j_1+1}^T\sum_{s=j_2+1}^T \Big( Y_{t,j_1}Y_{s,j_2}-E(Y_{t,j_1}Y_{s,j_2}) \Big)K\Big(\frac{t-s}{b_T}\Big)  \Big| = O_{P}\big(M_T^{8/m}b_T/\sqrt{T}\big).\]
\end{proof}

\begin{proof}[Proof of  Theorem~\ref{th.MPBoot}]
 Let  $ \widetilde{\xi}_k$, $k=1,2, \ldots, N_T,$  be Gaussian random variables  as in   Theorem~\ref{t.conf-bands}.  By the triangular inequality, 
\begin{align}
& \sup _{x\in \mathbb R} \Big\vert P\big( \max _{k=1,\ldots , N_T} 
		\sqrt{\frac{T}{M_T}}  \frac{\vert \widehat f_T(\lambda_{k,T})   - E\widehat f_T(\lambda_{k,T})\vert }{\widehat f_T(\lambda_{k,T})} \le x \big) 
			 -  P^\ast\big(   \max _{k=1,\ldots , N_T} \vert \xi^\ast_k \vert \le x \big) \Big\vert \nonumber \\
& \leq 	\sup _{x\in \mathbb R} \Big\vert P\big( \max _{k=1,\ldots , N_T} 
		\sqrt{\frac{T}{M_T}}  \frac{\vert \widehat f_T(\lambda_{k,T})   - E\widehat f_T(\lambda_{k,T})\vert }{\widehat f_T(\lambda_{k,T})} \le x \big) 
			-  P\big(   \max _{k=1,\ldots , N_T} \vert \widetilde{\xi}_k \vert \le x \big) \Big\vert \nonumber \\
& + \sup _{x\in \mathbb R} \Big\vert  P\big(   \max _{k=1,\ldots , N_T} \vert\widetilde{\xi}_k \vert \le x \big)  -  P^\ast\big(   \max _{k=1,\ldots , N_T} \vert \xi^\ast_k \vert \le x \big) \Big\vert.	
\end{align}
The first term is bounded  by   Theorem~\ref{t.conf-bands}.  To handle  the second term, we first show that 
\begin{align} \label{eq.DeltaBould}
&\Delta_T:=\max_{1\leq k_1,k_2\leq N_T}\big|  \widehat{C}_T(k_1,k_2)- C_T(k_1,k_2)\big|  \nonumber \\
 = & \ O_P\Big( \max_{k=1,\ldots , N_T} |\widehat{f}_T(\lambda_{k,T}) - f(\lambda_{k,T})| + \frac{b_TM_T}{\sqrt T}+   \frac{M_T}{b_T} \, +\,  M_Tb_T T^{8 a_s/m-1/2} \Big). 
\end{align}
We have
\begin{align*}
\Delta_T \leq &  \max_{1\leq k_1, k_2 \leq N_T} \Big| \Big(\frac{1}{\widehat{f}_T(\lambda_{k_1,T})\widehat{f}(\lambda_{k_2,T})} - \frac{1}{f_T(\lambda_{k_1,T})f(\lambda_{k_2,T})}\Big)\\
& \times\sum_{j_1=1}^{M_T} \sum_{j_2=1}^{M_T} a_{k_1,j_1}a_{k_2,j_2} \sigma_T(j_1,j_2)\Big| 
+  \max_{1\leq k_1, k_2 \leq N_T}\Big| \frac{1}{\widehat{f}_T(\lambda_{k_1,T})\widehat{f}(\lambda_{k_2,T})}\\
& \times  \sum_{j_1=1}^{M_T}\sum_{j_2=1}^{M_T} a_{k_1,j_1}a_{k_2,j_2}\big( \widehat{\sigma}_T(j_1,j_2) - \sigma_T(j_1,j_2)\big)\Big|\\
= &\ \Delta_{1,T} + \Delta_{2,T},
\end{align*}
with an obvious notation for $ \Delta_{1,T}$ and $ \Delta_{2,T}$.

For $ \Delta_{1,T}$ notice that by the boundedness of the spectral density $f$ and Assumption 2 
we get that
\begin{align*}
 &  \max_{1\leq k_1,k_2\leq N_T}  \Big|\frac{1}{\widehat{f}_T(\lambda_{k_1,T})\widehat{f}(\lambda_{k_2,T})} - \frac{1}{f_T(\lambda_{k_1,T})f(\lambda_{k_2,T})}\Big| \\
    \leq \ &  \Big(\frac{1}{\min_{1\leq k\leq N_T}\widehat{f}_T(\lambda_{k,T}) } \Big)^2\Big(\frac{1}{\min_{1\leq k\leq N_T}f_T(\lambda_{k,T}) }\Big)^2\\
   & \ \  \times \big( \max_{1\leq k\leq N_T}f(\lambda_{k,T})  +  \max_{1\leq k\leq N_T}\widehat{f}(\lambda_{k,T}) \big)    
   \max_{1\leq k \leq N_T}\big|\widehat{f}_T(\lambda_{k,T}) - f(\lambda_{k,T})\big|  \\
   &=  {\mathcal O}_P(    \max_{1\leq k\leq N_T}\big|\widehat{f}_T(\lambda_{k,T}) - f(\lambda_{k,T})\big|  ).
\end{align*}
Furthermore,  recalling the definition of the Gaussian random variables $ \xi_k$, we have 
\begin{align*}
& \max_{1\leq k_1,k_2\leq N_T} \Big|\sum_{j_1=1}^{M_T} \sum_{j_2=1}^{M_T} a_{k_1,j_1}a_{k_2,j_2} \sigma_T(j_1,j_2)\Big| 
 =  \max_{1\leq k_1,k_2\leq N_T} 
| E(\xi_{k_1}, \xi_{k_2})| \\
\leq &\  \big(\max_{1\leq k\leq N_T} \|\xi_k\|_2\big)^2 ={\mathcal O}(1);
\end{align*} 
see Remark 1. Hence $ \Delta_{1,T} ={\mathcal O}_P\big(     \max_{1\leq k \leq N_T}\big|\widehat{f}_T(\lambda_{k,T}) - f(\lambda_{k,T})\big|\big)$.

For $ \Delta_{2,T}$ we have 
\begin{align*}
\Delta_{2,T} & \leq \max_{1\leq j_1,j_2 \leq M_T}\big|\widehat{\sigma}_T(j_1,j_2) - \sigma_T(j_1,j_2) \Big| \Big( \frac{1}
{\min_{1\leq k \leq N_T}\widehat{f}_T(\lambda_{k,T})} \Big)^2 \\
& \ \ \  \ \times \Big(\max_{1\leq k\leq N_T}  \sum_{j=0}^{M_T} a_{k,j}\Big)^2\\
& = {\mathcal O}_P\Big( \frac{M_Tb_T}{\sqrt{T}} + \frac{M_T}{b_T} + M_Tb_T T^{8a_s /m -1/2}\Big),
\end{align*}
by Proposition~\ref{l.sigma-hat} and the fact that $  \sum_{j=0}^{M_T} a_{k,j} =O(\sqrt{M_T})$ uniformly in $1\leq k \leq N_T$. This establishes  (\ref{eq.DeltaBould}).

Recall next  that  for $ 1\leq k \leq N_T$, 
\[  E|\xi_{k}|^2  = f^2(\lambda_{k,T}) \int_{-1}^1 w^2(u) du +o(1) ,\]
and that $ E|\widetilde{\xi}_k|^2 =E|\xi_{k}|^2/f^2(\lambda_{k,T})$. From these expressions and by Assumption~2 and the boundedness  of $f$  we get  that $  \min_{1\leq k \leq N_T}C_T(k,k)$ and  $\max_{1\leq k\leq N_T}C_T(k,k)$ are, for $T$ large enough, bounded from below and from above, respectively, by positive constants. 
Using  Lemma A.1 of Zhang et al. (2022) we  then get
\begin{align*}
\sup _{x\in \mathbb R} \Big\vert  P\big(   \max _{k=1,\ldots , N_T} \vert\widetilde{\xi}_k \vert \le x \big)  -  & P^\ast\big(   \max _{k=1,\ldots , N_T} \vert \xi^\ast_k \vert \le x \big) \Big\vert   \\
= & {\mathcal O}_P\Big(\frac{\Delta_T^{1/6}}{(1+\log(N_T))^{1/4}} + \Delta_T^{1/3} \log^3(N_T)\Big)\\
= &  o_P(\Delta_T^{1/6}).
\end{align*}


\end{proof}

{\bf Acknowledgements}
The authors are grateful to  the Co-Editor and to two referees for their valuable comments that lead to an   improved  version of this paper. They are also thankful 
to Panagiotis Maouris and Alexander Braumann for  helping carrying out the numerical work of Section~\ref{sec.Sim}.



\end{document}